\documentclass[11pt,amssymb]{amsart}
\usepackage{graphicx}
\usepackage{amsmath,amssymb,amsfonts,euscript,enumerate}
\usepackage{amsthm}
\usepackage{color,wrapfig}
\usepackage{pxfonts}
\usepackage{verbatim}
\newtheorem{thm}{Theorem}[section]
\newtheorem{cor}[thm]{Corollary}
\newtheorem{lemma}[thm]{Lemma}

\newtheorem{definition}{Definition}[section]

\newcommand{\R}{{\mathbb{R}}}

\newcommand{\e}{\epsilon}
\newcommand{\vp}{\varphi}
\newcommand{\osc}{\operatornamewithlimits{osc}}

\newcommand{\D}{\nabla}

\newcommand{\La}{\triangle}
\newcommand{\bs}{\backslash}

\begin{document}

\title{H\"older Estimates for Singular Non-local Parabolic Equations}

\author[Sunghoon Kim]{Sunghoon Kim}

\address{Sunghoon Kim : School of Mathematical Sciences,
Seoul National University, San56-1 Shinrim-dong Kwanak-gu Seoul
151-747,South Korea } \email{ gauss79@snu.ac.kr }

\author[Ki-Ahm Lee]{Ki-Ahm Lee}

\address{Ki-Ahm Lee : School of Mathematical Sciences,
Seoul National University, San56-1 Shinrim-dong Kwanak-gu Seoul
151-747,South Korea } \email{ kiahm@math.snu.ac.kr }

\begin{abstract}
In this paper, we establish local H\"older estimate for non-negative
solutions of the singular equation \eqref{eq-nlocal-PME-1} below,
for $m$ in the range of exponents $(\frac{n-2\sigma}{n+2\sigma},1)$.
Since we have trouble in finding the local energy inequality of $v$
directly. we use the fact that the operator $(-\La)^{\sigma}$ can be
thought as the normal derivative of some extension $v^{\ast}$ of $v$
to the upper half space, \cite{CS}, i.e., $v$ is regarded as
boundary value of $v^{\ast}$ the solution of some local extension
problem. Therefore, the local H\"older estimate of $v$ can be
obtained by the same regularity of $v^{\ast}$. In addition, it
enables us to describe the behaviour of solution of non-local fast
diffusion equation near their extinction time.
\end{abstract}

\maketitle

\setcounter{equation}{0}
\setcounter{thm}{0}

\section{Introduction}

\setcounter{equation}{0}
\setcounter{thm}{0}

In this paper, we consider initial value problem with fractional
fast diffusion:
\begin{equation}\label{eq-frac-FDE}
\begin{cases}
\begin{aligned}
&(-\La)^{\sigma}u^m+u_t=0 \qquad \qquad\qquad \mbox{in
$\Omega$}\\
&\qquad \quad u=0 \qquad \qquad \qquad \qquad \quad \mbox{on $\R^n\bs\Omega$}\\
&\quad u(x,0)=u_0(x) \qquad \mbox{non-negative and
$\dot{H}^{\sigma}_0$-bounded}
\end{aligned}
\end{cases}
\end{equation}
in the range of exponents $\frac{n-2\sigma}{n+2\sigma}<m<1$, with
$0<\sigma<1$. The fractional Laplacian of a function $f:\R^n\to\R$
is expressed by the formula
\begin{equation*}
(-\La)^{\sigma}f(x)=C_{n,\sigma}\int_{\R^n}\frac{f(x)-f(y)}{|x-y|^{n+2\sigma}}\,dy
\end{equation*}
where $C_{n,\sigma}$ is some normalization constant. In addition,
the norm in $\dot{H}^{\sigma}$ is given precisely by
\begin{equation}\label{H-space}
\|f\|_{\dot{H}^{\sigma}}=\sqrt{\int_{\R^n}\int_{\R^n}\frac{|f(x)-f(y)|^2}{|x-y|^{n+2\sigma}}dxdy}.
\end{equation}
That is equivalent to
\begin{equation*}
\|f\|_{\dot{H}^{\sigma}}\cong \sqrt{\int_{\R^n}
|\xi|^{2\sigma}|\hat{f}(\xi)|^2d\xi}.
\end{equation*}
for Fourier transform of $f$ in $x$. Note that the Sobolev embedding
results say that $\dot{H}^{\sigma}\subset L^{2n/(n-2\sigma)}$ (Chap
V in \cite{St}). Indeed, $\dot{H}^{\sigma}$ is the space of
$L^{2n/(n-2\sigma)}$
functions for which \eqref{H-space} is integrable. \\
\indent Several regularity results in the non-linear theory of
elliptic and parabolic differential equations are based on H\"older
estimates for linear equations with measurable coefficients.
H\"older estimates can be drived from the Harnack inequalities as
the Theorem 8.22 in \cite{GT}. In \cite{De}, E. De Giorgi proved a
H\"older estimate for second order uniformly elliptic equations in
divergence form with measurable coefficients, and the regularity of
minimizers for non-linear convex functionals followed from there.
Chen and DiBenedetto study in \cite{YD2} the question of H\"older
estimates of solutions of singular parabolic equations with
measurable coefficients of the $p$-Laplacian type, and the methods
apply to the porous medium equations in the fast diffusion range.
The Harnack inequality for non-negative solutions of singular
parabolic equations is proved in \cite{YD3}. The equation may have
bounded measurable coefficients and have the form
$u_t-(a_{ij}(x,t)|u|^{m-1}u_{x_i})_{x_j}=0$ with
$\frac{n-2\sigma}{n+2\sigma}<m<1$. The authors use an iteration
method in the context of quasi-linear singular parabolic equations
that is different from the classical
iteration techniques of E. De Giorgi \cite{De}. \\
\indent In the non-divergent case, the corresponding result was
obtained by Krylov and Safanov \cite{KS}, and it is an essential
tool in proving $C^{1,\beta}$
regularity for fully nonlinear elliptic equations. \\
\indent The regularity of harmonic functions with respect to
non-local operators was studied in several recent papers like
\cite{BK} and \cite{BL}, however their point of view is
probabilistic. The pure analytic point of view for regularity
appears in \cite{Si}, which deals with non-divergence structure.
Caffarelli and Vasseur \cite{CV} show that, using Digiorgi type
iterative techniques, a certain class of weak solutions of the
quasi-geostrophic equation with initial $L^2$ data and critical
diffusion $(-\La)^{\frac{1}{2}}$ gain local H\"older regularity for
any space dimension. In \cite{CV}, they consider the harmonic
extension problem corresponding to the original one in order to
avoid difficulties stem from the non-locality. \\
\indent We are interested in studying of the properties of
non-linear eigenvalue problems with fractional powers: regularities,
geometric properties, etc. However, there are lots of difficulties
to be solved directly. For example, the geometric properties of the
first eigenfunction for the fractional Laplacian
$-(-\La)^{\frac{b}{2}}$ is sill unsolved. More precisely, the
conjecture 1.1 in \cite{RTP}, \emph{let $\vp_1^{b}$ be the ground
state eigenfunction for the symmetric stable processes of index
$0<b<2$ killed upon leaving the interval
$I=(-1,1)$. Then is $\vp_1^{b}$ concave on $I$?}\\
\indent Many methods is devised to treat non-linear eigenvalue
problems. One of them is parabolic method, \cite{LV}, which relies
on the fact that the non-linear eigenvalue problem can be described
as
the asymptotic profile of a corresponding parabolic flow.\\
\indent For this reason, in this work, we will deal with a H\"older
regularity of $v=u^m$, which is a solution of
\begin{equation}\tag{M.P}\label{eq-nlocal-PME-1}
\begin{cases}
\begin{aligned}
(-\La)^{\sigma}v&+\left(v^{\frac{1}{m}}\right)_t=0 \qquad
\qquad\qquad \mbox{in
$\Omega$}\\
v&=0 \qquad \qquad \qquad \qquad \quad \mbox{on
$\R^n\bs\Omega$}\\
v(x,0)&=v_0(x)=u_0^m(x) \qquad \qquad \mbox{in $\Omega$},
\end{aligned}
\end{cases}
\end{equation}
assuming that the initial value $v_0$ is strictly positive in the
interior of $\Omega$ in $\R^n$. Main two Theorems state as follows:
\begin{thm}(Boundedness for positive times)\label{thm-L-infty}
\item Let $v(x,t)$ be a function in
$L^{\infty}(0,T;L^{\frac{2n}{n-2\sigma}}(\Omega))\cap L^2(0,T;
\dot{H}^{\sigma}_0(\R^n))$, then
\begin{equation*}
\sup_{x\in\Omega}|v(x,T)|\leq
C^{\ast}\frac{\|v_0\|_{L^{\frac{2n}{n-2\sigma}}(\Omega)}}{T^{\frac{mn}{2mn-(n-2\sigma)(1+m)}}}
\end{equation*}
for some constant $C^{\ast}>0$.
\end{thm}
For the second theorem, we need better control of $v$.
\begin{thm}[H\"older regularity of fractional FDE]\label{thm-main}
\item For $x_0=(x_0^1,\cdots,x_0^n)$, we define $Q_r(x_0,t_0)=[x_0^i-r,x_0^i+r]^n\times[t_0-r^{2\sigma},t_0]$, for $t_0>r^{2\sigma}>0$.
Assume now that $[x_0^i-r,x_0^i+r]^n\subset\Omega$ and $v(x,t)$ is
bounded in $\R^n\times[t_0-r^{2\sigma},t_0]$, then there exist
constants $\gamma$ and $\beta$ in $(0,1)$ that can be determined a
priori only in terms of the data, such that $v$ is $C^{\beta}$ in
$Q_{\gamma r}(x_0,t_0)$.
\end{thm}
In order to develop the H\"older regularity method, it is necessary
to localize the energy inequality by space and time truncation. Due
to the non-locality of the diffusion, this appears complicated. On
the other hand, $(-\La)^{\sigma}v$ can be thought as the normal
derivative of some extension of $v$ (the Dirichlet to Neumann
operator of $v$. See \cite{CS} for a general discussion). This
allows us to realize the truncation as a standard local equation in
one more dimension: we introduce first the corresponding extension
$v^{\ast}$ defined from $C_0^{\infty}(\R^n)$ to
$C_0^{\infty}(\R^n\times\R^+)$ by:
\begin{equation*}
\begin{aligned}
&-\D (y^a\D v^*)=0 \qquad \qquad \mbox{in $\R^n\times(0,\infty)$}\\
&\quad v^{\ast}(x,0)=v(x) \qquad \qquad \mbox{for $x\in\R^n$}
\end{aligned}
\end{equation*}
for $a=1-2\sigma$. (This extension consists simply in convolving $v$
with the Poisson kernel of the upper half space in one more
variable.) Then the following result holds true: for $v$ defined on
$\R^n$, we have:
\begin{equation*}
(-\La)^{\sigma}v(x)=\partial_{\nu}v^{\ast}(x,0)=-\lim_{y\to0}y^av^{\ast}_y(x,y)
\end{equation*}
where we denote $\partial_{\nu}v^{\ast}$ the outward normal
derivative of $v^{\ast}$ on the boundary $\{y=0\}$. Hence, it is
possible to consider the solution $v$ of problem
\eqref{eq-nlocal-PME-1} as the boundary value of $v^{\ast}$ which is
solution of
\begin{equation}\label{eq-Holder}
\begin{cases}
\begin{aligned}
\D (y^a\D v^*)&=0 \qquad \qquad \qquad \qquad \text{in $y>0$}\\
\lim_{y\to0}y^av^{\ast}_y(x,y,t)=&(v^{\frac{1}{m}})_t(x,0,t) \qquad
\qquad x\in\Omega\\
v^{\ast}(x,0,t)&=0 \qquad \qquad \qquad \qquad \mbox{on
$\R^n\bs\Omega$}.
\end{aligned}
\end{cases}
\end{equation}
Thus, we can obtain the H\"older estimate of $v$ immediately by
showing the H\"older regularity of $v^{\ast}$.\\
\indent Since the diffusion coefficients $D(v)=|v|^{1-\frac{1}{m}}$
goes to infinity as $v\to 0$, we need to control the oscillation of
$v$ from below. Hence, we consider the new function $w^{\ast}$
derived from $v^{\ast}$ such that
$w^{\ast}(x,y,t)=M-v^{\ast}(x,y,t+t_0)$ with $M=M(t_0)=\sup_{t\geq
t_0>0} v^{\ast}$. By Theorem \ref{thm-L-infty}, we know that the
solution satisfies
\begin{equation*}
v^{\ast}(\cdot,t)\leq M(t_0)<\infty \qquad \quad (t\geq t_0).
\end{equation*}
From this, we get to a familiar situation:
\begin{equation}\label{eq-Hoilder(M-v)}
\begin{cases}
\begin{aligned}
\nabla(y^a\nabla w^{\ast})&=0 \qquad \qquad \qquad \qquad \qquad \quad \mbox{in $y>0$}\\
-\lim_{y\to
0^+}y^a\nabla_yw^{\ast}(x,y)=&\left[(M-w^{\ast})^{\frac{1}{m}}\right]_t(x,0)
\qquad \qquad x\in\Omega\\
w^{\ast}(x,0,t)&=M \qquad \qquad \qquad \qquad \qquad \mbox{on
$\R^n\bs\Omega$}.
\end{aligned}
\end{cases}
\end{equation}
\\
\indent The paper is divided into four parts: In Section 2 we study
several properties of the Fast Diffusion Equation (shortly, FDE)
with fractional powers
\begin{equation}\label{eq-frac-FDE-1}
(-\La)^{\sigma}u^m+u_t=0, \qquad
\left(\frac{n-2\sigma}{n+2\sigma}<m<1\right).
\end{equation}
More precisely, we explain \emph{Scale Invariance},
\emph{$L^1$-Contraction} and \emph{Extinction on Finite Time}. In
Section 3, we show the existence of weak solution of the problem
\eqref{eq-nlocal-PME-1}. Also, we investigate the boundedness of the
solutions of problem \eqref{eq-nlocal-PME-1} for positive times.
Lastly in this section, we compute local energy inequality of
$(w^{\ast}-k)_\pm$ which will be a key step in establishing local
H\"older estimates. The proof of the H\"older regularity of problems
is given in Section 4. In this section, we consider the extension
$v^{\ast}$ of $v$ that solves \eqref{eq-nlocal-PME-1}. This allows
us to treat non-linear problems, involving fractional Laplacians, as
a local problems. In the last section, we study the existence of
non-linear eigenvalue problem with fractional powers which is
asymptotic profile of the parabolic flow \eqref{eq-nlocal-PME-1} on
extinction time.
\\
\\
\textbf{Notations:} Before we explain the main ideas of the paper,
let us summarize the notations and definitions that will be used.
\begin{itemize}
\item Numbers: $\frac{n-2\sigma}{n+2\sigma}<m<1$, $\alpha=1-\frac{1}{m}$, $0<\sigma<1$, $a=1-2\sigma$ and $M=\sup_{t\geq t_0} v^{\ast}$ for some $t_0>0$.
\item For $x_0=(x_0^1,\cdots,x_0^n)\in\R^n$, we denote by
$B_r(x_0)=[x_0^i-r,x_0^i+r]^n$ a cube in the $x$ variable only, and
$B^{\ast}_r(x_0)=B_r(x_0)\times(0,r)\in \R^n\times(0,\infty)$ a cube
in the $x$, $y$ variables sitting on the plane $y=0$.
\item We construct the cylinders
\begin{equation*}
Q_r(x_0,t_0)=B_r(x_0)\times(t_0-r^{2\sigma},t_0)
\end{equation*}
and, for $\alpha=1-\frac{1}{m}$,
\begin{equation*}
\begin{aligned}
Q_r(\omega)(x_0,t_0)=B_r(x_0)\times\left(t_0-\frac{r^{2\sigma}}{\omega^{\alpha}},t_0\right).
\end{aligned}
\end{equation*}
In addition, for $\rho_1$ and $\rho_2$,
\begin{equation*}
Q_r(\omega,\rho_1,\rho_2)(x_0,t_0)= B_{r-\rho_1 r}(x_0)\times
\left(t_0-(1-\rho_2)\frac{r^{2\sigma}}{\omega^{\alpha}},t_0\right).
\end{equation*}
\item The cylinders $Q^{\ast}_r(x_0,t_0)$, $Q^{\ast}_r(\omega)(x_0,t_0)$ and
$Q^{\ast}_r(\omega,\rho_1,\rho_2)(x_0,t_0)$ are obtained in a
similar manner by replacing $B_r(x_0)$ and $B_{r-\rho_1 r}(x_0)$ by
$B^{\ast}_r(x_0)$ and $B^{\ast}_{r-\rho_1 r}(x_0)$ respectively.
\item We let $Q_r$, $Q^{\ast}_r$, $Q_r(\omega)$, $Q^{\ast}_r(\omega)$,
$Q_r(\omega,\rho_1,\rho_2)$ and $Q^*_r(\omega,\rho_1,\rho_2)$ be the
cylinders around the point $(x_0,t_0)=(0,0)$.
\end{itemize}
\indent Let us start with showing the properties of the solution for
FDE with fractional powers in the following section.
\section{Properties of Fast Diffusion Equations with Fractional Powers}

\setcounter{equation}{0}
\setcounter{thm}{0}

Since the operator $(-\La)^{\sigma}$ converges to $(-\La)$ as the
quantity $\sigma$ goes to $1$, it is natural to expect that the
solutions of the equation \eqref{eq-frac-FDE-1} has a lot in common
with those of the FDE (of course, not in complete accord). Hence,
before coming to main issue, we will discuss such properties of FDE
in this section.
\subsection{Scale Invariance} Let us examine the application of
\emph{scaling transformations} to the fractional powers of the FDE
in some detail. Let $u=u(x,t)$ be a solution of the fractional
powers of the FDE,
\begin{equation}\label{eq-frac-PME}
(-\La)^{\sigma} u^m+u_t=0 \qquad \qquad (0<m<1).
\end{equation}
We apply the group of dilations in all the variables
\begin{equation*}
u'=Ku, \qquad x'=Lx \qquad t'=Tt,
\end{equation*}
and impose the condition that $u'$ so expressed as a function of
$x'$ and $t'$, i.e.,
\begin{equation*}
u'(x',t')=Ku(\frac{x'}{L},\frac{t'}{T}),
\end{equation*}
has to be again a solution of (\ref{eq-frac-PME}). Then:
\begin{equation}\label{sol-invari}
\frac{\partial u'}{\partial t'}=\frac{K}{T}\frac{\partial
u}{\partial t}\left(\frac{x'}{L},\frac{t'}{T}\right)
\end{equation}
and
\begin{equation*}
\begin{aligned}
(-\La)^{\sigma}
[u'(x',t')]^m&=\int\frac{[u'(x',t')]^m-[u'(y',t')]^m}{|x'-y'|^{n+2\sigma}}dy'\\
&=\frac{K^m}{L^{2\sigma}}\int
\frac{\big[u\left(\frac{x'}{L},\frac{t'}{T}\right)\big]^m-\big[u\left(\frac{y'}{L},\frac{t'}{T}\right)\big]^m}{|\frac{x'}{L}-\frac{y'}{L}|^{n+2\sigma}}d\left(\frac{y'}{L}\right)\\
&=\frac{K^m}{L^{2\sigma}}(-\La)^{\sigma} u^m.
\end{aligned}
\end{equation*}
Hence, (\ref{sol-invari}) will be a solution if and only if
$KT^{-1}=K^{m}L^{-2\sigma}$, i.e.,
\begin{equation*}
K^{m-1}=L^{2\sigma}T^{-1}.
\end{equation*}
We thus obtain a two-parametric transformation group acting on the
set of solution of (\ref{eq-frac-PME}). Assuming that $m \neq 1$, we
may choose as free parameters $L$ and $T$, so that it can be written
as
\begin{equation*}
u'(x',t')=L^{\frac{2\sigma}{m-1}}T^{-\frac{1}{m-1}}u\left(\frac{x'}{L},\frac{t'}{T}\right)
\end{equation*}
Using standard letters for the independent variables as putting
$u'=\tau u$, we get:
\begin{equation}\label{sol-tau-invari}
(\tau
u)(x,t)=L^{\frac{2\sigma}{m-1}}T^{-\frac{1}{m-1}}u\left(\frac{x}{L},\frac{t}{T}\right).
\end{equation}
The conclusion is:
\begin{lemma}
If $u$ is a solution of the fractional powers of the FDE in a
certain class of solutions $\mathbb{S}$ that is closed under
dilations in $x,t$ and $u$, then $\tau u$ given by
(\ref{sol-tau-invari}) is again a solution of the fractional powers
of the FDE in the same class.
\end{lemma}
\subsection{$L^1$-contraction} This is a very important estimate
which has played a key role in the fractional powers of the FDE
theory. It will allow us to develop existence, uniqueness and
stability theory in the space $L^1$.
\begin{lemma}[$L^1$-contraction]\label{L^1-contraction}
Let $\Omega$ is a bounded domain of $\R^n$ with smooth boundary, and
let $u$ and $\tilde{u}$ be two smooth solutions of the fractional
powers of the fast diffusion equation (FDE):
\begin{equation*}
\begin{cases}
\begin{aligned}
&(-\La)^{\sigma}u^m+u_t=0 \qquad \qquad \mbox{in
$Q_{T}=\Omega\times(0,T)$}\\
&\qquad \quad u=0 \qquad \qquad \qquad   \qquad \mbox{on
$\R^n\bs\Omega$}
\end{aligned}
\end{cases}
\end{equation*}
with initial date $u_0$, $\tilde{u}_0$ respectively. We have for
every $t>\tau \geq 0$
\begin{equation}\label{L^1-con-1}
\int_{\Omega}\big[u(x,t)-\tilde{u}(x,t)\big]_+dx \leq
\int_{\Omega}\big[u(x,\tau)-\tilde{u}(x,\tau)\big]_+dx
\end{equation}
As a consequence,
\begin{equation}\label{L^1-con-2}
\|u(t)-\tilde{u}(t)\|_1\leq \|u_0-\tilde{u}_0\|_1.
\end{equation}
\end{lemma}
\begin{proof}
Let $p \in C^1(\R)$ be such that $0\leq p \leq 1, p(s)=0$ for $s
\leq 0,\,\, p'(s)>0$ for $s>0$. Let $w=u^m-\tilde{u}^m$ which
vanishes on the boundary $\R^n\bs\Omega\times[0,T)$. Subtracting the
equations satisfied by $u$ and $\tilde{u}$, multiplying by $p(w)$
and integrating in $\R^n$, we have for
\begin{equation*}
\begin{aligned}
\int_{\R^n}(u-\tilde{u})_t p(w)dx&=\int_{\R^n}-(-\La)^{\sigma}wp(w)dx\\
&=\int_{\R^n}\int_{\R^n}\frac{\big[-w(x,t)+w(y,t)\big]p(w(x,t))}{|x-y|^{n+2\sigma}}dydx.
\end{aligned}
\end{equation*}
Since
\begin{equation*}
\begin{aligned}
A=&\int_{\R^n}\int_{\R^n}\frac{\big[-w(x,t)+w(y,t)\big]p(w(x,t))}{|x-y|^{n+2\sigma}}dydx\\
=&\int_{\R^n}\int_{\R^n}\frac{\big[-w(x,t)+w(y,t)\big]\big[p(w(x,t))-p(w(y,t))\big]}{|x-y|^{n+2\sigma}}dydx\\
&+\int_{\R^n}\int_{\R^n}\frac{\big[-w(x,t)+w(y,t)\big]p(w(y,t))}{|x-y|^{n+2\sigma}}dydx\\
=&\int_{\R^n}\int_{\R^n}\frac{\big[-w(x,t)+w(y,t)\big]\big[p(w(x,t))-p(w(y,t))\big]}{|x-y|^{n+2\sigma}}dydx-A,
\end{aligned}
\end{equation*}
we obtain
\begin{equation*}
\begin{aligned}
\int_{\R^n}(u&-\tilde{u})_t p(w)dx\\
&=\frac{1}{2}\int_{\R^n}\int_{\R^n}\frac{\big[-w(x,t)+w(y,t)\big]\big[p(w(x,t))-p(w(y,t))\big]}{|x-y|^{n+2\sigma}}dydx.
\end{aligned}
\end{equation*}
Note that the term in the right-hand side is non-positive. Therefore
letting $p$ converge to the sign function $\textrm{sign}_0^+$,
\cite{Va}, and observing that
\begin{equation*}
\frac{\partial}{\partial t}(u-\tilde{u})_+=(u-\tilde{u})_t
\textrm{sign}_0^+,
\end{equation*}
we get
\begin{equation*}
\frac{d}{dt}\int_{\Omega}(u-\tilde{u})_+dx=\frac{d}{dt}\int_{\R^n}(u-\tilde{u})_+dx
\leq 0,
\end{equation*}
which implies (\ref{L^1-con-1}) for $u, \tilde{u}$. To obtain
(\ref{L^1-con-2}), combine (\ref{L^1-con-1}) applied first to $u$
and $\tilde{u}$ and then to $\tilde{u}$ and $u$.
\end{proof}

\subsection{Extinction in Finite Time} The main difference with porous medium equation is the
finite time convergence of the solutions to the zero solution, which
replaces the infinite time stabilization that holds for $m\geq 1$.
This phenomenon is called \emph{extinction in finite time} and read
as follows.
\begin{lemma}\label{lem-fde-extinction}
If $u(x,t)$ is the $C^{2,1}$ solution of the fast diffusion equation
with fractional powers :
\begin{equation*}
\begin{cases}
\begin{aligned}
&(-\La)^{\sigma}u^m+u_t=0 \qquad \qquad \mbox{in
$Q_{\infty}=\Omega\times(0,\infty)$}\\
&\qquad \quad u=0 \qquad \qquad \qquad \qquad \qquad  \mbox{on $\R^n\bs\Omega$}\\
&\quad u(x,0)=u_0(x)\in C^0(\Omega)
\end{aligned}
\end{cases}
\end{equation*}
where $\Omega$ is a bounded domain of $\R^n$ with smooth boundary,
then there exists $T^{\ast}>0$ such that $u(\cdot,t)=0$ for all
$t\geq T^{\ast}$, i.e.,
\begin{equation*}
\lim_{t\to T^{\ast}}\|u(\cdot,t)\|_{\infty}=0
\end{equation*}
for some $T^{\ast}>0$. The solution can be continued past the
extinction time $T^{\ast}$ in a weak sense as $u \equiv 0$.
\end{lemma}
\begin{proof}
It is enough to construct a super-solution $V$ with the property of
extinction. We will choose the function $V$ in the form
$V(x,t)=X(x)T(t)$. Let $R>0$ be such that $\Omega\subset B_R(0)$,
and let us define a function $X(x)$ by
\begin{equation*}
X(x)=\begin{cases}
       \frac{1}{R^{\frac{n-2\sigma}{m}}} \qquad \,\,|x|\leq R\\
       \frac{1}{|x|^{\frac{n-2\sigma}{m}}} \qquad |x|>R
     \end{cases}
\end{equation*}
and a function $T(t)$ by
\begin{equation*}
T(t)=\begin{cases}
       C(T^{\ast}-t)^{\frac{1}{1-m}} \qquad \,\,t\leq T^{\ast}\\
       \qquad 0 \qquad \qquad \qquad t>T^{\ast}
     \end{cases}
\end{equation*}
for a constant $C>0$ we can choose later. Since
\begin{equation*}
\begin{aligned}
(-\La)^{\sigma}X^m &\geq\int_{\R^n\bs
B_{2R}(0)}\frac{X^m(x)-X^m(y)}{|x-y|^{n+2\sigma}}\,dy\\
&\geq
\left(\frac{2}{3}\right)^{n+2\sigma}\left(1-\frac{1}{2^{n-2\sigma}}\right)\frac{1}{R^{n-2\sigma}}\int_{\R^n\bs
B_{2R}(0)}\frac{1}{|y|^{n+2\sigma}}\,dy\\
&\geq
\frac{1}{2\sigma}\left(\frac{2}{3}\right)^{n+2\sigma}\left(\frac{1}{2^{\sigma}}-\frac{1}{2^{n}}\right)\frac{1}{R^{n}}
\end{aligned}
\end{equation*}
for all $x\in \Omega$, we have
\begin{equation*}
\begin{aligned}
&(-\La)^{\sigma}V^m+V_t\\
&\quad \geq
C^m(T^{\ast}-t)^{\frac{m}{1-m}}\left[\frac{1}{2\sigma}\left(\frac{2}{3}\right)^{n+2\sigma}\left(\frac{1}{2^{\sigma}}-\frac{1}{2^{n}}\right)\frac{1}{R^{n}}-\frac{C^{1-m}}{(1-m)R^{\frac{n-2\sigma}{m}}}\right]\\
&\quad =0
\end{aligned}
\end{equation*}
if we choose
$C=\left[\left(\frac{1-m}{2\sigma}\right)\left(\frac{2}{3}\right)^{n+2\sigma}\left(\frac{1}{2^{\sigma}}-\frac{1}{2^{n}}\right)R^{\frac{(1-m)n-2\sigma}{m}}\right]^{\frac{1}{1-m}}$.
\end{proof}
Next we deal with the two Lemmas. The ideas are based on the proof
of Lemma 1 and 2 in \cite{BH} respectively.
\begin{lemma}[Estimates on Finite Extinction Time]\label{lem-Estimates-Finite-Extinction-Time}
When $\frac{n-2\sigma}{n+2\sigma}<m<1$, there exists a positive
constant $C$ such that the solution $v=u^m$ of
\eqref{eq-nlocal-PME-1} satisfies
\begin{equation*}
T^{\ast}-t \leq C
\left(\int_{\R^n}v^{\frac{m+1}{m}}(x,t)\,\,dx\right)^{\frac{1-m}{1+m}}.
\end{equation*}
\end{lemma}
\begin{proof}
Multiplying equation \eqref{eq-nlocal-PME-1} by $v$ and integrating
by parts, we obtain the inequality
\begin{equation}\label{eq-finite-extinction-1}
\frac{d}{dt}\left(\int_{\R^n}v^{\frac{m+1}{m}}\,\,dx\right)=-(m+1)\int_{\R^n}v(-\La)^{\sigma}v\,\,dx=-(m+1)\|v\|^2_{\dot{H}^{\sigma}}.
\end{equation}
By the compactness of imbedding, there are constants $C$ and $C'$
such that, for any $v\in \dot{H}_0^{\sigma}$,
\begin{equation}\label{eq-finite-extinction-2}
\|v\|^2_{\dot{H}^{\sigma}}\geq C\|v\|^2_{\frac{2n}{n-2\sigma}} \geq
C'\|v\|^2_{\frac{m+1}{m}}
\end{equation}
when $\frac{n-2\sigma}{n+2\sigma}<m<1$. The last inequality makes
use of H\"older inequality. Substituting
\eqref{eq-finite-extinction-2} into \eqref{eq-finite-extinction-1}
one obtains an inequality which can be integrated to yield
\begin{equation*}
\left(\int_{\R^n}v^{\frac{m+1}{m}}(x,t')\,\,dx\right)^{\frac{1-m}{1+m}}-\left(\int_{\R^n}v^{\frac{m+1}{m}}(x,t)\,\,dx\right)^{\frac{1-m}{1+m}}\leq
-C'(1-m)(t'-t)
\end{equation*}
for $T^{\ast}>t'\geq t$. Letting $t'\to T^{\ast}$ and multiplying by
$-1$, we have the statement of the Lemma.
\end{proof}
\begin{lemma}\label{lem-lemma-2}
When $\frac{n-2\sigma}{n+2\sigma}<m<1$, the solution $v$ of
\eqref{eq-nlocal-PME-1} satisfies
\begin{equation}\label{eq-lemma-2}
\int_{\R^n}v^{\frac{m+1}{m}}(x,t)\,dx\leq
\left(1-\frac{t}{T^{\ast}}\right)^{\frac{1+m}{1-m}}\int_{\R^n}v^{\frac{m+1}{m}}(x,0)\,dx.
\end{equation}
\end{lemma}
\begin{proof}
Since $v(x,t)=0$ on $\R^n\bs\Omega$, applying the Cauchy-Schwarz
inequality, we obtain
\begin{equation*}
\left(\int_{\R^n}v(-\La)^{\sigma}v\,dx\right)^2=\left(\int_{\Omega}v(-\La)^{\sigma}v\,dx\right)^2\leq
\int_{\Omega}v^{\frac{m+1}{m}}\,dx\cdot\int_{\Omega}\frac{\left[(-\La)^{\sigma}v\right]^2}{v^{\frac{1-m}{m}}}\,dx.
\end{equation*}
It is convenient to rewrite this as
\begin{equation}\label{eq-basic-inequality}
\frac{\int_{\R^n}v(-\La)^{\sigma}v\,dx}{\int_{\Omega}v^{\frac{m+1}{m}}\,dx}\leq
\frac{\int_{\Omega}\frac{\left[(-\La)^{\sigma}v\right]^2}{v^{\frac{1-m}{m}}}\,dx}{\int_{\R^n}v(-\La)^{\sigma}v\,dx}.
\end{equation}
Next an expression for
$\frac{d}{dt}\int_{\R^n}v(x,t)(-\La)^{\sigma}v(x,t)\,dx$ will be
derived. Carrying out the indicated differentiation, one obtains
\begin{equation*}
\begin{aligned}
&\frac{d}{dt}\int_{\R^n}v(x,t)(-\La)^{\sigma}v(x,t)\,dx=2\int_{\R^n}v_t(x,t)(-\La)^{\sigma}v(x,t)\,dx\\
&\qquad \qquad
=2\int_{\Omega}v_t(x,t)(-\La)^{\sigma}v(x,t)\,dx=-2m\int_{\Omega}\frac{\left[(-\La)^{\sigma}v(x,t)\right]^2}{v^{\frac{1-m}{m}}(x,t)}\,dx.
\end{aligned}
\end{equation*}
The second equality follows from the fact that $v_t(x,t)=0$ on
$\R^n\bs\Omega$. Using this expression together with
\eqref{eq-finite-extinction-1}, we see that the basic inequality
\eqref{eq-basic-inequality} can be rewritten as
\begin{equation*}
-\frac{1}{(m+1)}\frac{\frac{d}{dt}\int_{\R^n}v^{\frac{m+1}{m}}(x,t)\,\,dx}{\int_{\R^n}v^{\frac{m+1}{m}}(x,t)\,dx}\leq
-\frac{\frac{d}{dt}\int_{\R^n}v(x,t)(-\La)^{\sigma}v(x,t)\,dx}{2m\int_{\R^n}v(x,t)(-\La)^{\sigma}v(x,t)\,dx}.
\end{equation*}
Integrating this differential inequality, we obtain
\begin{equation}\label{eq-integrating}
\left(\frac{\int_{\R^n}v^{\frac{m+1}{m}}(x,t)\,dx}{\int_{\R^n}v^{\frac{m+1}{m}}(x,s)\,dx}\right)^{\frac{2m}{m+1}}\geq
\frac{\int_{\R^n}v(x,t)(-\La)^{\sigma}v(x,t)\,dx}{\int_{\R^n}v(x,s)(-\La)^{\sigma}v(x,s)\,dx},
\end{equation}
for $T^{\ast}\geq t>s\geq 0$. Define
\begin{equation*}
Z(t)=\left(\int_{\R^n}v^{\frac{m+1}{m}}(x,t)\,dx\right)^{\frac{1-m}{1+m}}.
\end{equation*}
Then \eqref{eq-integrating} is just $Z'(s)\leq Z'(t)$ for all $s\leq
t$, so that $Z''\geq 0$ and $Z$ is convex. Hence
\begin{equation*}
Z(t)\geq Z(T)+\left(\frac{Z(T)-Z(s)}{T-s}\right)(T-t)
\end{equation*}
for all $s\leq t\leq T$. Setting $s=0$ and $T=T^{\ast}$, one obtains
\eqref{eq-lemma-2}.
\end{proof}

\section{Weak solutions and Local Energy Inequality}

\setcounter{equation}{0}
\setcounter{thm}{0}

First, we study the problem (\ref{eq-nlocal-PME-1}) in the class of
nonnegative weak solutions. For the remainder of this paper we
assume that
\begin{equation*}
\frac{n-2\sigma}{n+2\sigma}<m<1
\end{equation*}
holds.
\subsection{Weak solutions and Existence}
\begin{definition}
A non-negative weak solution of equation (\ref{eq-nlocal-PME-1}) is
a locally integrable function, $v \in
L^1_{loc}(\R^n\times[0,\infty))$, such that $(-\La)^{\sigma}v \in
L^1_{loc}(\R^n\times[0,\infty))$ and $v=0$ on $\R^n\bs\Omega$, and
the identity
\begin{equation}\label{eq-very-weak-solution}
\int_{0}^{\infty}\int_{\Omega}v^{\frac{1}{m}}\eta_t\,\,dxdt=\int_{0}^{\infty}\int_{\Omega}
v[-(-\La)^{\sigma}\eta] \,\,dxdt
\end{equation}
holds for any test function $\eta \in
C^{2,1}_c(\Omega\times[0,\infty))$ with $\eta=0$ on $\R^n\bs\Omega$.
\end{definition}
We show the existence and comparison result for weak solutions. the
proof is similar to the proof of the Theorem (5.5) in \cite{Va}.
\begin{lemma}
There exists a non-negative weak solution of
(\ref{eq-nlocal-PME-1}). Moreover, the comparison principle holds
for these solutions: $v$, $\hat{v}$ are weak solutions with initial
data such that $v_0 \leq \hat{v}_0$ a.e. in $\Omega$, then $v \leq
\hat{v}$ a.e. for all $t>0$.
\end{lemma}
\begin{proof}
We first assume that $v_0$ is a non-negative smooth function with
compact support in $\Omega$. Then we begin by constructing a
sequence of approximate initial data $v_{0n}$ which does not take
the value zero, so as to avoid the degeneracy of the equation. We
may simple put
\begin{equation*}
v_{0n}(x)=v_0(x)+\frac{1}{n} \qquad \mbox{and} \qquad N=\sup(v_0).
\end{equation*}
and we solve the problem
\begin{equation}\label{eq-exist-weaksolution-assume-smooth}
\begin{cases}
\begin{aligned}
(-\La)^{\sigma}v_n&+\left(v_n^{\frac{1}{m}}\right)_t=0 \qquad
\mbox{in
\,$\Omega\times[0,\infty)$}\\
v_n(x,0)&=v_{0n}(x) \qquad \mbox{in \,$\Omega$}\\
v_n(x,t)&=1/n \qquad x\in\R^n\bs\Omega.
\end{aligned}
\end{cases}
\end{equation}
Then, the $L^1$-contraction implies that
\begin{equation*}
\frac{1}{n}\leq v_n(x,t)\leq N+\frac{1}{n}.
\end{equation*}
Therefore, we are dealing with a uniformly parabolic problem. Hence
the problem (\ref{eq-exist-weaksolution-assume-smooth}) has a unique
smooth solution $v_n$. Moreover, using the $L^1$-contraction again
\begin{equation*}
v_{n+1}(x,t) \leq v_n(x,t)
\end{equation*}
for all n $\geq 1$. Hence, we may define the function
\begin{equation*}
v(x,t)=\lim_{n\to\infty}v_n(x,t)
\end{equation*}
as a monotone limit of bounded non-negative functions. We see that
$v_n$ converges to $v$ in $L^p(\Omega\times[0,\infty))$ for every
$1\leq p<\infty$. In order to show that this $v$ is the weak
solution of problem (\ref{eq-nlocal-PME-1}), we multiply the
equation (\ref{eq-exist-weaksolution-assume-smooth}) by a test
function $\eta$ and integrate by part, the Theorem 1.2 in \cite{Gu},
in $\Omega\times[0,\infty)$ to obtain
\begin{equation*}
\int_{0}^{\infty}\int_{\Omega}\left(v_n^{\frac{1}{m}}\right)\eta_t\,\,dxdt=\int_{0}^{\infty}\int_{\Omega}
v_n[-(-\La)^{\sigma}\eta] \,\,dxdt.
\end{equation*}
Letting $n\to\infty$, we get the identity
(\ref{eq-very-weak-solution}) for $v$. Also $v_n=1/n$ on
$\R^n\bs\Omega$ and $0\leq v\leq v_n$, we have
\begin{equation*}
v(x,t)=0 \qquad \mbox{on \,$\R^n\bs\Omega$}.
\end{equation*}
Therefore, $v$ is a weak solution of
(\ref{eq-nlocal-PME-1}). \\
\indent If we have initial data $v_0$ and $\hat{v}_0$ such that
$v_0\leq \hat{v}_0$, then the above approximation process produces
ordered approximating sequences, $v_{0n}\leq\hat{v}_{0n}$. By the
$L^1$-contraction, we have $v_n\leq\hat{v}_n$ for every $n\geq 1$.
Letting $n\to\infty$, we have $v\leq\hat{v}$.\\
\indent We next assume that $v_0$ is bounded. In order to apply the
method used in above, we also construct a sequence of approximate
initial data $v_{0n}\in C^{2,1}(\Omega)$ with $v=1/n$ on
$\R^n\bs\Omega$. Then we can obtain the approximate solutions
$v_n\in C^{2,1}(\Omega)$ with $v=1/n$ on $\R^n\bs\Omega$. By the
$L^1$-contraction, $v_n\to v$ in $C([0,\infty):L^1(\Omega))$ as
$n\to\infty$ and $v=0$ on $\R^n\bs\Omega$. Since the functions are
bounded, convergence also takes place in $C([0,\infty):L^p(\Omega))$
for all $p \leq \infty$. Hence the proof ends as before. \\
\indent Finally, we assume that the initial date $v_0$ is
integrable. By considering the sequence of approximations of the
initial data
\begin{equation*}
v_{0k}(x)=\min\big\{v_0(x),k\},
\end{equation*}
we can complete the proof as before.
\end{proof}

\subsection{$L^{\infty}$ bounds}
This subsection is devoted to the proof of Theorem
\ref{thm-L-infty}. The simple proof is based on a recurrence
non-linear relation between consecutive truncations of $v$ at an
increasing sequence of levels. Following the similar arguments as in
Section 2 in \cite{CV}, we can get a desired result.
\begin{proof}[Proof of Theorem
\ref{thm-L-infty}] We use the energy inequality for the levels
\begin{equation*}
C_k=N\left(1-\frac{1}{2^{k}}\right)
\end{equation*}
where $N$ will be chosen later. Multiplying the equation in
\eqref{eq-nlocal-PME-1} by the function $v_k=\left(v-C_k\right)_+$
and integrating in space, $\R^n$, we have
\begin{equation}\label{eq-L-infty-bound-1}
\frac{1}{m}\int_{\R^n}\frac{d}{dt}\left[\int_0^{v_k}\left(\xi+C_k\right)^{\frac{1}{m}-1}\xi\,d\xi\right]\,dx+\int_{\R^n}v_k\left[(-\La)^{\sigma}v_k\right]\,dx
\leq 0
\end{equation}
since
\begin{equation*}
(C_k+v_k)^{\frac{1}{m}}v_k=(C_k+\xi)^{\frac{1}{m}}\xi\Big|^{\xi=v_k}_{\xi=0}=\int_0^{v_k}\frac{d}{d\xi}\left[(C_k+\xi)^{\frac{1}{m}}\xi\right]\,d\xi.
\end{equation*}
Let us fix a $t_0>0$, we want to show that $v$ is bounded for
$t>t_0$. For
\begin{equation*}
T_k=t_0\left(1-\frac{1}{2^{k}}\right),
\end{equation*}
we integrate \eqref{eq-L-infty-bound-1} in time between $s$,
$T_{k-1}<s<T_k$, and $t>T_k$ and between $s$ and $+\infty$. Then we
find
\begin{equation*}
\begin{aligned}
\sup_{t\geq
T_k}\int_{\R^n}\left[\int_0^{v_k}\left(\xi+C_k\right)^{\frac{1}{m}-1}\xi\,d\xi\right]\,dx&+m\int_{T_k}^{\infty}\|v_k\|^2_{\dot{H}^{\sigma}}\,dt\\
&\leq
2\int_{\R^n}\left[\int_0^{v_k}\left(\xi+C_k\right)^{\frac{1}{m}-1}\xi\,d\xi\right](s)\,dx.
\end{aligned}
\end{equation*}
This leads to
\begin{equation*}
\sup_{t\geq
T_k}\int_{\Omega}v_k^{\frac{1}{m}+1}\,dx+(m+1)\int_{T_k}^{\infty}\|v_k\|^2_{\dot{H}^{\sigma}}\,dt\leq
(m+1)\int_{\Omega}v^{\frac{1}{m}-1}(s)v^2_k(s)\,dx
\end{equation*}
since $v_k$ has compact support in $\Omega$. In addition, H\"older
inequality gives
\begin{equation*}
\begin{aligned}
\int_{\Omega}v^{\frac{1}{m}-1}(s)v^2_k(s)\,dx&\leq
\left(\int_{\Omega}v^{\frac{1-m}{m}\cdot\frac{1+m}{1-m}}(s)\,dx\right)^{\frac{1-m}{1+m}}\left(\int_{\Omega}v_k^{2\cdot\frac{1+m}{2m}}(s)\,dx\right)^{\frac{2m}{1+m}}\\
&\leq
C\left(\int_{\Omega}v_k^{\frac{1+m}{m}}(s)\,dx\right)^{\frac{2m}{1+m}}
\end{aligned}
\end{equation*}
for some constant $C>0$. Hence, for the level set of energy
\begin{equation*}
U_k=\sup_{t\geq
T_k}\int_{\Omega}v_k^{\frac{1}{m}+1}\,dx+(m+1)\int_{T_k}^{\infty}\|v_k\|^2_{\dot{H}^{\sigma}}\,dt,
\end{equation*}
we have
\begin{equation*}
U_k\leq
2C(m+1)\left(\int_{\Omega}v_k^{\frac{1+m}{m}}(s)\,dx\right)^{\frac{2m}{1+m}}.
\end{equation*}
Taking the mean value in $s$ on $[T_{k-1},T_k]$, we find
\begin{equation*}
\begin{aligned}
U_k&\leq
\frac{C2^{k+1}(m+1)}{t_0}\int^{T_k}_{T_{k-1}}\left(\int_{\Omega}v_k^{\frac{1+m}{m}}\,dx\right)^{\frac{2m}{1+m}}\,dt\\
&\leq
C'\left(\frac{2^k}{t_0}\right)^{1-\frac{1-m}{1+m}}\left(\int^{T_k}_{T_{k-1}}\int_{\Omega}v_k^{\frac{1+m}{m}}\,dxdt\right)^{\frac{2m}{1+m}}\\
&\leq
C'\left(\frac{2^k}{t_0}\right)^{1-\frac{1-m}{1+m}}\left(\int^{\infty}_{T_{k-1}}\int_{\Omega}v_k^{\frac{1+m}{m}}\,dxdt\right)^{\frac{2m}{1+m}}
\end{aligned}
\end{equation*}
for some constant $C'=C'(\Omega,m,n,\sigma)$. We want to control the
right hand side by $U_{k-1}$. Sobolev and interpolation inequalities
give (See interpolation inequalites of $L^p$-space in Lemma
\ref{eq-Sobolev-Inequality}):
\begin{equation*}
U_{k-1}\geq
C\|v_{k-1}\|^2_{L^{2\left(1+\frac{(1+m)\sigma}{mn}\right)}([T_{k-1},\infty)\times\R^n)}.
\end{equation*}
Note that if $v_k>0$, then $v_{k-1}\geq \frac{N}{2^k}$. Thus
\begin{equation*}
\textbf{1}_{\{v_k>0\}}\leq
\left(\frac{2^k}{N}v_{k-1}\right)^{2\left(1+\frac{(1+m)\sigma}{mn}\right)-\frac{1+m}{m}}.
\end{equation*}
Hence
\begin{equation*}
\begin{aligned}
U_k&\leq
C'\left(\frac{2^k}{t_0}\right)^{1-\frac{1-m}{1+m}}\left(\int^{\infty}_{T_{k-1}}\int_{\Omega}v_{k-1}^{\frac{1+m}{m}}\cdot\textbf{1}_{\{v_k>0\}}\,dxdt\right)^{\frac{2m}{1+m}}\\
&\leq
C'\frac{\left(2^{\frac{4m}{1+m}\left(1+\frac{(1+m)\sigma}{mn}\right)-\frac{2}{1+m}}\right)^k}{t_0^{1-\frac{1-m}{1+m}}N^{\frac{4m}{1+m}\left(1+\frac{(1+m)\sigma}{mn}\right)-2}}\left(\int_{\Omega}v_{k-1}^{2\left(1+\frac{(1+m)\sigma}{mn}\right)}\,dxdt\right)^{\frac{2m}{1+m}}\\
&\leq
C'\frac{\left(2^{\frac{4m}{1+m}\left(1+\frac{(1+m)\sigma}{mn}\right)-\frac{2}{1+m}}\right)^k}{t_0^{1-\frac{1-m}{1+m}}N^{\frac{4m}{1+m}\left(1+\frac{(1+m)\sigma}{mn}\right)-2}}U_{k-1}^{\frac{2m}{1+m}\left(1+\frac{(1+m)\sigma}{mn}\right)}
\end{aligned}
\end{equation*}
Note that $\frac{2m}{1+m}\left(1+\frac{(1+m)\sigma}{mn}\right)>1$.
Thus, for $N$ such that $N^2t_0^{\frac{2mn}{2mn-(n-2\sigma)(1+m)}}$
big enough (depending on $U_0$), Lemma 4.1 of Chap. I in \cite{Di} ,
we have $U_k$ which converges to zero. This gives $v\leq N$ for
$t\geq t_0$. Hence we come to the conclusion using the fact that
$U_0\leq \|v(\cdot,0)\|^2_{L^{\frac{2n}{n-2\sigma}}}$.
\end{proof}
\subsection{Local Energy Estimate of $w^{\ast}$}
The rest of this section is devoted to the proof of the Sobolev and
local energy inequalities for the extension
$w^{\ast}(x,y,t)=M-v^{\ast}(x,y,t+t_0)$ with $M=\sup_{t\geq t_0>0}
v^{\ast}$. The effect of the non-local part of $(-\La)^{\sigma}$
becomes encoded locally in the extra variable. The first result,
Sobolev inequality, states as follows:
\begin{lemma}[Sobolev Inequality]\label{eq-Sobolev-Inequality}
For a cut-off  function $\eta$ compactly supported in $B_r$,
\begin{equation}\label{Sobolev-eq-1}
\|\eta\, v\|_{L^{\frac{2n}{n-2\sigma}}(\R^n)}\leq C\|\eta\,
v\|_{\dot{H}^{\sigma}(\R^n)}
\end{equation}
and
\begin{equation}
\begin{aligned}
\|\eta\, v\|^2_{L^2(t_1,t_2;L^{2}(\R^n))}\leq \,\,&\\
C\Big(\sup_{t_1\leq t\leq t_2}\|\eta\, v\|^2_{L^{2}(\R^n)}&+\|\D
(\eta\, v)^{\ast}\|^2_{L^2(t_1,t_2;L^2(B^*_r, y^a))}\Big)\,
|\{\eta\, v>0\}|^{\frac{2\sigma}{n+2\sigma}}
\end{aligned}
\end{equation}
for some $C>0$.
\end{lemma}
\begin{proof}
The first inequality is well known result. The Sobolev embedding
results say that $\dot{H}^{\sigma}\subset L^{2n/(n-2\sigma)}$ (See,
for example, Chapter V in \cite{St}).\\
\indent To prove the second inequality, let $\chi_{\eta v}(x,t)$ be
the function with
\begin{equation*}
\chi_{_{\eta v}}=\begin{cases}
                     \begin{array}{ccc}
                           1 & \qquad &\eta v > 0\\
                           0 & \qquad &\eta v=0,
                     \end{array}
                \end{cases}
\end{equation*}
then we have
\begin{equation*}
\|\eta v\|^2_{L(t_1,t_2;L^2(\R^n)}=\int_{t_1}^{t_2}\int_{\R^n}|\eta
v|^2dxdt=\int_{t_1}^{t_2}\int_{\R^n}|\eta v|^2\chi_{_{\eta v}}dxdt.
\end{equation*}
Thus, by the H\"older inequality, we obtain
\begin{equation*}
\begin{aligned}
\|\eta v\|^2_{L(t_1,t_2;L^2(\R^n)}&\leq
\int_{t_1}^{t_2}\left(\int_{\R^n}|\eta
v|^{2\cdot\left(\frac{n+2\sigma}{n}\right)}dx\right)^{\frac{n}{n+2\sigma}}\left(\int_{\R^n}\left(\chi_{_{\eta
v}}\right)^{\frac{n+2\sigma}{2\sigma}}dx\right)^{\frac{2\sigma}{n+2\sigma}}dt\\
&\leq \left(\int_{t_1}^{t_2}\int_{\R^n}|\eta
v|^{2\cdot\left(\frac{n+2\sigma}{n}\right)}dxdt\right)^{\frac{n}{n+2\sigma}}\left(\int_{t_1}^{t_2}\int_{\R^n}\chi_{_{\eta
v}}dxdt\right)^{\frac{2\sigma}{n+2\sigma}}\\
&=\left(\int_{t_1}^{t_2}\int_{\R^n}|\eta
v|^{2\cdot\left(\frac{n+2\sigma}{n}\right)}dxdt\right)^{\frac{n}{n+2\sigma}}|\{\eta
v>0\}|^{\frac{2\sigma}{n+2\sigma}}.
\end{aligned}
\end{equation*}
Now we use interpolation inequalities of $L^p$ spaces,
\begin{equation*}
\begin{aligned}
&\left(\int_{t_1}^{t_2}\int_{\R^n}|\eta
v|^{2\cdot\left(\frac{n+2\sigma}{n}\right)}dxdt\right)^{\frac{n}{n+2\sigma}}\leq\\
&\qquad \qquad \qquad \qquad
\left[\int_{t_1}^{t_2}\left(\int_{\R^n}|\eta
v|^2dx\right)^{(1-\beta)p}\left(\int_{\R^n}|\eta
v|^{2q}dx\right)^{\frac{\beta p}{q}}dt\right]^{\frac{1}{p}}
\end{aligned}
\end{equation*}
where $1<p=\frac{n+2\sigma}{n}<q$ and
$\frac{1}{p}=\frac{\beta}{q}+\frac{1-\beta}{1}$
$\left(\beta=\frac{1}{p}\right)$.\\
Thus
\begin{equation*}
\begin{aligned}
&\left(\int_{t_1}^{t_2}\int_{\R^n}|\eta
v|^{2\cdot\left(\frac{n+2\sigma}{n}\right)}dxdt\right)^{\frac{n}{n+2\sigma}}\leq\\
&\qquad \qquad \qquad \qquad \sup_{t_1\leq t\leq
t_2}\left(\int_{\R^n}|\eta
v|^2dx\right)+\int_{t_1}^{t_2}\left(\int_{\R^n}|\eta
v|^{2q}dx\right)^{\frac{1}{q}}dt.
\end{aligned}
\end{equation*}
where $q=\frac{n}{n-2\sigma}$. From the first Sobolev inequality
(\ref{Sobolev-eq-1}), we have
\begin{equation*}
\begin{aligned}
\left(\int_{\R^n}|\eta v|^{2q}dx\right)^{\frac{1}{q}}&\leq C\|\eta
v\|^2_{\dot{H}^{\sigma}}=C\int_{\R^n}\eta v(-\La)^{\sigma}(\eta
v)dx\\
&=C\int_0^{\infty}\int_{\R^n}y^a|\nabla(\eta v)^{\ast}|^2dxdy,
\end{aligned}
\end{equation*}
which gives the desired result.
\end{proof}
Next, we will derive local energy estimate in the interior of
$\Omega\times(0,\infty)$ which will be the main tools in
establishing local H\"older estimates for the solutions. Assuming
that the cube $B_r=[-r,r]^n\subset\Omega$.
\begin{lemma}[Local Energy Estimate]\label{lem-Loc-Ene-Est-1}
\item Let $t_1$, $t_2$ be such that $t_1<t_2$ and let $v^{\ast}\in L^{\infty}(t_1,t_2;L^2(\R^n\times\R^+))$ be
solution to \eqref{eq-Holder} and let
$w^{\ast}(x,y,t)=M-v^{\ast}(x,y,t+t_0)$ with $M=\sup_{t\geq t_0>0}
v^{\ast}$. Then, there exists a constant $\lambda$ such that for
every $t_1\leq t\leq t_2$ and cut-off function $\eta$ such that the
restriction of $\eta(w^{\ast}-k)_\pm$ on $B^{\ast}_r$ is compactly
supported in $B_r\times(-r,r)$:
\begin{equation}\label{Holder-energy-estimate}
\begin{aligned}
&\frac{1}{m}\int_{B_r\times
  \{t_2\}}\eta^2\left[\int_0^{(w-k)_\pm}(M-k\mp\xi)^{\frac{1}{m}-1}\xi\,
d\xi\right]\, dx\\
&\qquad+\int_{t_1}^{t_2}\int_{B^*_r}|\D (\eta (w^*-k)_\pm)|^2\,y^adx\,dy\,dt\\
&\leq \int_{t_1}^{t_2}\int_{B^*_r} |(\D
\eta)(w^*-k)_\pm|^2\,y^adx\,dy\,dt\\
&\qquad+\frac{2}{m}\int_{t_1}^{t_2}\int_{B_r}
\left[\int_0^{(w-k)_\pm}(M-k\mp\xi)^{\frac{1}{m}-1}\xi\,
d\xi\right]|\eta\eta_t|\, dx\,dt\\
&\qquad+\frac{1}{m}\int_{B_r\times\{t_1\}}
\eta^2\left[\int_0^{(w-k)_\pm}(M-k\mp\xi)^{\frac{1}{m}-1}\xi\,
d\xi\right]\, dx
\end{aligned}
\end{equation}
\end{lemma}
\begin{proof}
We have for every $t_1<t<t_2$:
\begin{equation*}
\begin{aligned}
0&=\int_{B^{\ast}_r}\eta^2(w^{\ast}-k)_+\nabla(y^a\nabla
w^{\ast})\,\,dxdy\\
&=-\int_{B^{\ast}_r}y^a\eta\nabla\big(\eta(w^{\ast}-k)_+\big)\cdot
\nabla w^{\ast}\,\,dxdy-\int_{B^{\ast}_r}y^a\eta(w^{\ast}-k)_+\nabla
\eta\cdot \nabla w^{\ast}\,\,dxdy\\
& \qquad \qquad \qquad +\int_{B_r}\eta^2(w-k)_+(-\La)^{\sigma}w
\,\,dx\\
&=-\int_{B^{\ast}_r}y^a|\nabla\big(\eta(w^{\ast}-k)_+\big)|^2\,dxdy+\int_{B^{\ast}_r}y^aw^{\ast}\nabla\big(\eta(w^{\ast}-k)_+\big)\cdot\nabla\eta\,dxdy\\
&\qquad -\int_{B^{\ast}_r}y^a\eta(w^{\ast}-k)_+\nabla \eta\cdot
\nabla w^{\ast}\,\,dxdy+\int_{B_r}\eta^2(w-k)_+(-\La)^{\sigma}w
\,dx\\
&=-\int_{B^{\ast}_r}y^a|\nabla\big(\eta(w^{\ast}-k)_+\big)|^2\,\,dxdy+\int_{B^{\ast}_r}y^a\big((w^{\ast}-k)_+\big)^2|\nabla\eta|^2\,dxdy\\
& \qquad \qquad \qquad +\int_{B_r}\eta^2(w-k)_+(-\La)^{\sigma}w
\,dx.
\end{aligned}
\end{equation*}
Using the equation (\ref{eq-Hoilder(M-v)}), we find that
\begin{equation*}
\begin{aligned}
&\int_{B_r}\eta^2(w-k)_+(-\La)^{\sigma}w
\,dx=\int_{B_r}\eta^2(w-k)_+\left[(M-w)^{\frac{1}{m}}\right]_t\,dx\\
&\qquad \quad=\int_{B_r}\eta^2(w-k)_+\left[(M-k-(w-k)_+)^{\frac{1}{m}}\right]_t\,dx\\
&\qquad \quad=\int_{B_r}\eta^2\Big[(w-k)_+(M-k-(w-k)_+)^{\frac{1}{m}}\\
&\qquad \qquad \quad +\frac{m}{m+1}(M-k-(w-k)_+)^{\frac{m+1}{m}}-\frac{m}{m+1}(M-k)^{\frac{m+1}{m}}\Big]_t\,dx\\
\end{aligned}
\end{equation*}
Since
\begin{equation*}
\begin{aligned}
(M-k-(w-k)_+)^{\frac{1}{m}}(w-k)_+&=(M-k-\xi)^{\frac{1}{m}}\xi\Big|^{\xi=(w-k)_+}_{\xi=0}\\
&=\int^{(w-k)_+}_{0}\frac{d}{d\xi}\left[(M-k-\xi)^{\frac{1}{m}}\xi\right]\,d\xi,
\end{aligned}
\end{equation*}
we get the desired result from
\begin{equation*}
\begin{aligned}
(w-k)_+(M-&k-(w-k)_+)^{\frac{1}{m}}+\frac{m}{m+1}(M-k-(w-k)_+)^{\frac{m+1}{m}}\\
&-\frac{m}{m+1}(M-k)^{\frac{m+1}{m}}=-\frac{1}{m}\int^{(w-k)_+}_{0}(M-k-\xi)^{\frac{1}{m}-1}\xi\,\,d\xi.
\end{aligned}
\end{equation*}
The estimate of $(w-k)_-$ can be obtained in a similar manner.
\end{proof}
\section{H\"older Regularity of $w^{\ast}$}\label{subsec-holder-fdf}

\setcounter{equation}{0}
\setcounter{thm}{0}

From now on, we begin the story of H\"older continuity of $w^{\ast}$
the solution of \eqref{eq-Hoilder(M-v)}. In order to develop the
H\"older regularity method, we get over the two humps: non-local
diffusion and degeneracy. We first follow Caffarelli and Vasseur's
ideas \cite{CV} to solve a difficulty stemming from non-linear
evolution equations with fractional diffusion. Also we use the
technique developed in \cite{Di}, \cite{DK} and \cite{YD1} to
overcome the degeneracy of equation. \\
\indent The key idea of the proof is to work with cylinders whose
dimensions are suitably rescaled to reflect the degeneracy exhibited
by the equation. To make this precise, fix
$(x_0,t_0)\in\Omega\times(0,T]$, for some $T>0$, and construct the
cylinder
\begin{equation*}
Q_{2R}(x_0,t_0)\subset \Omega\times(0,T].
\end{equation*}
After a translation we may assume that $(x_0,t_0)=(0,0)$. Set
\begin{equation*}
\mu^+=\sup_{Q^{\ast}_{2R}}w^{\ast},\qquad
\mu^-=\inf_{Q^{\ast}_{2R}}w^{\ast},\qquad
\omega=\osc_{Q^{\ast}_{2R}}w^{\ast}=\mu^+-\mu^-
\end{equation*}
and construct the cylinder
\begin{equation*}
Q^{\ast}_R(\theta_0)=B_R^{\ast}\times\left(-\frac{R^{2\sigma}}{\theta_0^{\alpha}},0\right),
\qquad \qquad \theta_0=\frac{\omega}{A}, \quad \alpha=1-\frac{1}{m}
\end{equation*}
where $A$ is a constant to be determined later only in terms of  the
data. We will assume that
\begin{equation*}
\left(\frac{\omega}{A}\right)^{\alpha}\geq 1.
\end{equation*}
This implies the inclusion
\begin{equation*}
Q^{\ast}_R(\theta_0)\subset Q^{\ast}_{2R}
\end{equation*}
and the inequality
\begin{equation*}
\osc_{Q^{\ast}_R(\theta_0)}w^{\ast}\leq \omega.
\end{equation*}
Then, the first alternative in this section states as follows:
\begin{lemma}\label{lem-Hoilder-1}
There exists positive numbers $\rho$ and $\lambda$ independent of
$\mu^{\pm}$ and $\omega$ such that if
\begin{equation}\label{eq-p-harnack-cond}
\left|\left\{(x,t)\in
  Q_R(\theta_0); w(x,t)>\mu^+-\frac{\omega}{2}\right\}\right|<\rho|Q_R(\theta_0)|
\end{equation}
then
\begin{equation*}
w(x,t)<\mu^+-\frac{\lambda\omega}{4}
\end{equation*}
for all $(x,t)\in Q_{\frac{R}{2}}(\theta_0)$.
\end{lemma}
\begin{proof}
{\bf Step 1. Useful barrier functions:} Consider the function
$h_1(x,y)$, defined by:
\begin{equation}\label{eq-h}
\begin{aligned}
&\quad \D(y^a\D h_1)=0\qquad \qquad\text{in $B_{1}^*$}\\
&\qquad h_1=1\qquad \qquad \quad \text{on $\partial B_{1}^* \cap\{y>0\}$}\\
&\qquad h_1=0\qquad \qquad \quad \text{on $y=0$}.
\end{aligned}
\end{equation}
Then, following directly the maximum principle, there is uniform
constant $0<\lambda<1/2$ such that
\begin{equation*}
h_1(x,y)\leq (1-2\lambda)\quad \text{on $B_{1/2}^*$}.
\end{equation*}
For the scaling invariance of the equation for $h_1$,
\begin{equation*}
h_{1,R}(x,y)=h_1\left(\frac{x}{R},\frac{y}{R}\right)
\end{equation*}
satisfies \eqref{eq-h} in $B_R^*$ and
\begin{equation*}h_{1,R}(x,y)\leq (1-2\lambda)\quad \text{on
$B_{R/2}^*$}.
\end{equation*}
{\bf Step 2. Notations for the induction:} Set, for any non-negative
integer $k$,
\begin{equation*}
R_k=\frac{R}{2}+\frac{R}{2^{k+1}}, \quad \mbox{and} \quad
l_k=\mu^+-\lambda\left(\frac{\omega}{4}+\frac{\omega}{2^{k+2}}\right).
\end{equation*}
We denote by $\tilde{B}^*_{R,\delta}$ the set $B_{R}\times
(0,\delta)$ and introduce the cylinders
\begin{equation*}
Q_{k}(\theta_0)=B_{\overline{R}_k}\times
\left(-\theta_0^{-\alpha}R_{k}^{2\sigma},0\right)
\end{equation*}
and
\begin{equation*}
\tilde{Q}^*_{R_k,\frac{\delta^k}{4}}(\theta_0)=\tilde{B}^*_{R_k,\frac{\delta^k}{4}}\times
\left(-\theta_0^{-\alpha}R_k^{2\sigma},0\right).
\end{equation*}
We also denote
\begin{equation*}
B(l_k,R_k)=\{(x,t)\in Q_{R_k}(\theta_0); w(x,t)>l_k\},
\end{equation*}
\begin{equation*}
B^*(l_k,R_k)=\{(x,y,t)\in Q^*_{R_k}(\theta_0); w^*(x,y,t)>l_k\}
\end{equation*}
and
\begin{equation*}
\tilde{B}^*\left(l_k,R_k,\frac{\delta^k}{4}\right)=\{(x,y,t)\in
\tilde{Q}^*_{R_k,\frac{\delta^k}{4}}(\omega); w^*(x,y,t)>l_k\}.
\end{equation*}
From above, we define:
\begin{equation*}
|B(l_k,R_k)|=\int^{0}_{-\theta_0^{-\alpha}R_k^{2\sigma}}|\{x\in
B_{R_k};\, w(x,t)>l_k\}|dt,
\end{equation*}
\begin{equation*}
|B^*(l_k,R_k)|=\int^{0}_{-\theta_0^{-\alpha}R_k^{2\sigma}}|\{(x,y)\in
B^*_{R_k};\, w^*(x,y,t)>l_k\}|dt
\end{equation*}
and
\begin{equation*}
\left|\tilde{B}^*\left(l_k,R_k,\frac{\delta^k}{4}\right)\right|=\int^{0}_{-\theta_0^{-\alpha}R_k^{2\sigma}}|\{(x,y)\in
\tilde{B}^*_{R_k,\frac{\delta^k}{4}};\, w^*(x,y,t)>l_k\}|dt.
\end{equation*}
{\bf  Step 3. Claim for the Induction:}
 We set
\begin{equation*}
w_k=(w-l_k)_+, \qquad w_k^{\ast}=(w^{\ast}-l_k)_+.
\end{equation*}
Note that $w_k^{\ast}\neq(w_k)^{\ast}$. We consider a cut-off
function $\eta_k (x,t)$ such that
\begin{equation}
\begin{aligned}
&\quad \qquad 0<\eta_k\leq 1 \qquad \qquad \mbox{in $Q_{R_k}(\theta_0)$}\\
&\quad \qquad \quad \eta_k=1 \qquad \qquad \quad \mbox{in $Q_{R_{k+1}}(\theta_0)$}\\
&\quad \qquad \quad \eta_k=0 \qquad \mbox{on the parabolic boundary
of
$Q_{R_k}(\theta_0)$}\\
&|\D \eta_k|\leq \frac{2^{k+2}}{R}, \qquad (\eta_k)_t\leq
\frac{2^{2\sigma(k+2)}\theta_0^{\alpha}}{R^{2\sigma}}\qquad
\theta_0=\frac{\omega}{A}.
\end{aligned}
\end{equation}
We will use the energy inequalities of Lemma
\eqref{lem-Loc-Ene-Est-1} written over the cylinders
$Q^{\ast}_{R_k}(\theta_0)$, for the function
$w^{\ast}_k=(w^{\ast}-l_k)_+$, where for $k=0,1,2,\cdots,$. Let
\begin{equation*}
Z_k=\theta_0^{\alpha}\,|B(l_k,R_k)|=\left(\frac{\omega}{A}\right)^{\alpha}\,|B(l_k,R_k)|,
\end{equation*}
then we are going to prove simultaneously that for every $k\geq 0$
\begin{equation}\label{ieq-z-k}
Z_k\leq N^{-k}
\end{equation}
for some constant $N>1$ and
\begin{equation}\label{eq-p-h-i2}
\eta_k w^*_k\quad\text{ is supported in $0\leq y\leq
\frac{\delta^k}{4}$}.
\end{equation}
{\bf Step 4. the contraction property of the support in $y$
direction:} We first want to show that $\eta_kv^*_k$ is supported in
$0\leq y\leq\frac{\delta^k}{4}$. By a comparison principle, we have:
\begin{equation*}
\begin{aligned}
\left(w^*-\left(\mu_+-\frac{\omega}{2}\right)\right)_+ \leq
\left[\left(w-\left(\mu_+-\frac{\omega}{2}\right)\right)_+\text{1}_{B_{R}}\right]\ast
P(y)+\frac{\omega}{2}h_{1,R}(x,y)
\end{aligned}
\end{equation*}
in $B_R^{\ast}\times\R^+$, where $P(y)$ is the Poisson kernel
introduced in Section 2.4 in \cite{CS}. Indeed, the right-hand side
function has the trace on the boundary is bigger than the one of
left-hand. Moreover:
\begin{equation*}
\begin{aligned}
&\left\|\left[\left(w-\left(\mu_+-\frac{\omega}{2}\right)\right)_+\text{1}_{B_{R}}\right]
\ast
P(y)\right\|_{L^{\infty}(y\geq \frac{R}4)} \\
&\qquad \qquad  \leq
C\Big\|P\Big(\frac{R}4\Big)\Big\|_{L^2}\left(\frac{\omega}{2}\right)\,
(\rho|Q_R(\theta_0)|)^{\frac{1}{2}}\leq C\sqrt{\rho}
\end{aligned}
\end{equation*}
Choosing $\rho$ small enough such that this constant is smaller that
$\frac{\lambda\omega}{2}$ gives:
\begin{equation*}
\left(w^*-\left(\mu_+-\frac{\omega}{2}\right)\right)_+ \leq
\frac{(1-\lambda)\omega}{2},
\end{equation*}
in $\{y>\frac{R}{4}\}\cap Q^*_{\frac{R}{2}}(\theta_0)$. Hence
\begin{equation*}
\left(w^*-\left(\mu_+-\frac{\lambda\omega}{2}\right)\right)_+ =0
\end{equation*}
in $\{y>\frac{R}{4}\}\cap Q^*_{\frac{R}{2}}(\theta_0)$. Since
$l_0=\mu_+-\frac{\lambda\omega}{2}$, we obtain that $\eta_0
w^*_0=\eta_1 (v^*-l_1)_-$ is supported in $0<y<\frac{\delta^0}{4}
R=\frac{R}{4}$ for $\delta_0=1$.\\
\indent Now we assume \eqref{eq-p-h-i2} is true at $k$th-step.  We
want to show that \eqref{eq-p-h-i2} is verified at $(k+1)$. First,
we will control $v^*_{k+1}$ in terms of $\eta_kv_k$ with some
controllable error, i.e., we will show also that the following is
verified at $k$:
\begin{equation}\label{quantity-level-1}
\eta_{k+1}w^{\ast}_{k+1}\leq \left[(\eta_kw_{k})\ast
P(z)\right]\eta_{k+1}, \qquad \qquad \mbox{on
$\tilde{B}^{\ast}_{R_k,\frac{\delta^k}{4}}$}
\end{equation}
where
$\tilde{B}^{\ast}_{R_k,\frac{\delta^k}{4}}=B_{R_k}\times(0,\frac{\delta^k}{4})$.
We consider now $h_2$ harmonic function defined by:
\begin{equation*}
\begin{aligned}
&\nabla\big(y^a\nabla h_2(z,y)\big)=0 \qquad \mbox{in $[0,\infty)\times[0,1]$}\\
&\qquad h_2(0,y)=1 \qquad \qquad \qquad 0\leq y\leq 1\\
&h_2(z,0)=h_2(z,1)=0 \qquad \mbox{for $0<x<\infty$}.
\end{aligned}
\end{equation*}
Then, there exists $C>0$ such that
\begin{equation*}
|h_2(z)|\leq Ce^{-\frac{z}{2}}.
\end{equation*}
Indeed, we can see that
\begin{equation*}
\begin{aligned}
&\quad h_2(z,y)\leq 2\sqrt{2}\cos
\left(\frac{y}{2}\right)e^{-\frac{z}{2}}\qquad\qquad (a\geq 0)\\
&h_2(z,y)\leq 4\sqrt{2}\sin
\left(\frac{y}{2}+\frac{\pi}{6}\right)e^{-\frac{z}{2}}\qquad\qquad
(a<0),
\end{aligned}
\end{equation*}
since this function is super-harmonic and bigger than $h_2$ on the
boundary. Consider
$B_{\frac{R_0}{2}+\frac{R_0}{2^{k+1+\frac{1}{2}}}}\times[0,\frac{\delta^k}{4}]$.
On $y=\frac{\delta^k}{4}$ we have no contribution thanks to the
induction property \eqref{eq-p-h-i2}. The contribution of the side
$y=0$ can be controlled by $(\eta_k w_k)\ast P(y)$. On each of the
other side, we control the contribution by the function of
\begin{equation*}
\frac{\omega}{2}\left[h_2\left(\frac{4(x_i-\vartheta^+)}{\delta^k},\frac{4y}{\delta^k}\right)+h_2\left(\frac{-4(x_i+\vartheta^+)}{\delta^k},\frac{4y}{\delta^k}\right)\right]
\end{equation*}
where $\vartheta^+=\frac{R_0}{2}+\frac{R_0}{2^{k+1+\frac{1}{2}}}$.
Since $h_2$ is super-solution and on the side $x_i=\vartheta^+$ and
$x_i=-\vartheta^+$ it is bigger than $1$, we have, by the maximum
principle:
\begin{equation*}
\begin{aligned}
&w^{\ast}_k\leq (\eta_kw_k)\ast P(y)\\
&\quad
+\frac{\omega}{2}\sum_{i=1}^{n}\left[h_2\left(\frac{4(x_i-\vartheta^+)}{\delta^k},\frac{4y}{\delta^k}\right)+h_2\left(\frac{-4(x_i+\vartheta^+)}{\delta^k},\frac{4y}{\delta^k}\right)\right].
\end{aligned}
\end{equation*}
For $x\in B_{R_{k+1}}$,
\begin{equation*}
\begin{aligned}
&\frac{\omega}{2}\sum_{i=1}^{n}\left[h_2\left(\frac{4(x_i-\vartheta^+)}{\delta^k},\frac{4y}{\delta^k}\right)+h_2\left(\frac{-4(x_i+\vartheta^+)}{\delta^k},\frac{4y}{\delta^k}\right)\right]\\
&\qquad \qquad \qquad \leq 2nC\omega
e^{-\frac{\sqrt{2}-1}{2^{k+3}\delta^k}} \leq \lambda 2^{-k-4}.
\end{aligned}
\end{equation*}
for some small constant $\delta$ and $\lambda$. This gives
(\ref{quantity-level-1}) since
\begin{equation*}
w^{\ast}_{k+1} \leq \left(w^{\ast}_k-\lambda2^{-k-3}\right)_+.
\end{equation*}
More precisely, this gives
\begin{equation*}
w^{\ast}_{k+1} \leq \left((\eta_kw_k)\ast
P(y)-\lambda2^{-k-4}\right)_+.
\end{equation*}
So
\begin{equation*}
\eta_{k+1}w^{\ast}_{k+1} \leq \left((\eta_kw_k)\ast
P(y)-\lambda2^{-k-4}\right)_+
\end{equation*}
Then, we find for $\frac{\delta^{k+1}}{4}\leq y \leq
\frac{\delta^k}{4}$,
\begin{equation*}
\begin{aligned}
|(\eta_kw_k)\ast P(y)| &\leq \sqrt{Z_k}\|P(y)\|_{L^2}\\
&\leq
\frac{\sqrt{\frac{\omega}{2}}N^{-\frac{k}{2}}}{\left(\frac{\delta}{4}\right)^{\frac{(k+1)n}{2}}}\|P(1)\|_{L^2}\leq
\lambda 2^{-k-4}
\end{aligned}
\end{equation*}
for large
$N>\sup\left(\frac{4^{n+1}}{\delta^{2}},\frac{2^{4n+9}\omega\|P(1)\|^2_{L^2}}{\lambda^2\delta^{2n}}\right)$.
Therefore,
\begin{equation*}
\eta_{k+1}w^{\ast}_{k+1} \leq 0 \qquad \qquad \mbox{for}\quad
\frac{\delta^{k+1}}{4}\leq y \leq \frac{\delta^k}{4}.
\end{equation*}
{\bf Step 5. Local Energy Estimate} First, we find the lower bound
of the following quantity
\begin{equation*}
D=\int_{B_r\times\{t_2\}}\eta^2\left[\int_0^{(w-k)_+}\left(M-k-\xi\right)^{\frac{1}{m}-1}\xi\,\,d\xi\right]dx.
\end{equation*}
Let's define the function $F(\xi)$ by
\begin{equation*}
F(\xi)=\left(M-k-\xi\right)^{\frac{1}{m}-1}\xi=\left(M-k-\xi\right)^{-\alpha}\xi
\qquad (0\leq \xi \leq M-k).
\end{equation*}
Then, we have
\begin{equation*}
\begin{aligned}
&\qquad F'(\xi)=\frac{1}{m}\left(M-k-\xi\right)^{-\alpha-1}\left(m(M-k)-\xi\right),\\
&F''(\xi)=-\frac{1}{m}\left(\frac{1}{m}-1\right)\left(M-k-\xi\right)^{-\alpha-2}\left(2m(M-k)-\xi\right).
\end{aligned}
\end{equation*}
Note that the sign change of second derivatives of $F(\xi)$ takes
place at $2m(M-k)$. If $(w-k)_+$ is smaller than $2m(M-k)$, we get
\begin{equation*}
\int_0^{(w-k)_+}F(\xi)\,\,d\xi \geq
\frac{1}{2}(w-k)_+F\left(\frac{(w-k)_+}{2}\right)
\end{equation*}
and
\begin{equation*}
D\geq
\frac{1}{4}\left(\frac{M-k}{2}\right)^{-\alpha}\int_{B_r\times\{t_2\}}\left[\eta(w-k)_+\right]^2\,\,dx.
\end{equation*}
since $F''(\xi)\leq 0$, $(0\leq \xi \leq (w-k)_+)$. On the other
hand,
\begin{equation*}
\begin{aligned}
\int_0^{(w-k)_+}F(\xi)\,\,d\xi &\geq \frac{1}{2}\cdot
2m(M-k)F(2m(M-k))\\
&=2m^2(M-k)^2\left[(1-2m)(M-k)\right]^{-\alpha}\\
&\qquad \geq 2m^2(w-k)_+^2\left[(1-2m)(M-k)\right]^{-\alpha}
\end{aligned}
\end{equation*}
when $(w-k)_+\geq 2m(M-k)$. Thus we obtain
\begin{equation*}
D\geq
2\left[(1-2m)(M-k)\right]^{-\alpha}\int_{B_r\times\{t_2\}}\left[\eta(w-k)_+\right]^2\,\,dx.
\end{equation*}
Therefore, there is a small constant $c>0$ such that
\begin{equation*}
\begin{aligned}
&c(M-k)^{-\alpha}\int_{B_r\times\{t_2\}}\left[\eta(w-k)_+\right]^2\,\,dx\\
&\qquad \quad \leq
\int_{B_r\times\{t_2\}}\eta^2\left[\int_0^{(w-k)_+}\left(M-k-\xi\right)^{\frac{1}{m}-1}\xi\,\,d\xi\right]dx.
\end{aligned}
\end{equation*}
Next, notice that from Step $4$, then we have
\begin{equation}\label{eq-convolution-1}
\begin{aligned}
&\|\eta_{k}\,w^*_{k}\|^2_{L^2(B^*_{R_k},|y|^a)}\leq
\|(\eta_{k-1}w_{k-1})*P(y)\|^2_{L^2(B^*_{R_{k-1}},|y|^a)}\\
&\qquad \qquad \leq \int_0^{\delta^{k-1}/4}\|(\eta_{k-1}w_{k-1})*(P(y))\|^2_{L^2(B_{R_{k-1}})}y^a\,dy\\
&\qquad \qquad \leq \|P(1)\|^2_{L^1(\R^n)}\|(\eta_{k-1}w_{k-1})\|^2_{L^2(B_{R_{k-1}})}\int_0^{\delta^{k-1}/4}y^a\,dy\\
&\qquad \qquad \leq
\frac{1}{1+a}\left(\frac{\delta^{k-1}}{4}\right)^{a+1}\|P(1)\|^2_{L^1(B_{R_{k-1}})}\|\eta_{k-1}w_{k-1}\|^2_{L^2(B_{R_{k-1}})}\\
&\qquad \qquad \leq \frac{1}{1+a}
\left(\frac{\delta^{k-1}}{4}\right)^{a+1}\frac{(\lambda\omega)^2}{4}|A_{l_{k-1},R_{k-1}}(t)|.
\end{aligned}
\end{equation}
where $A_{l,R}(t)= \{x\in B_R;\, w(x,t)>l \}$. We can apply the
Lemma \ref{lem-Loc-Ene-Est-1} (Local Energy Estimate) on
$\eta_kw^{\ast}_k\text{1}_{\{0<y<\frac{\delta^k}{4}\}}$
\begin{equation}
\begin{aligned}
&c\left(\frac{\lambda\omega}{4}\right)^{-\alpha}\sup_{-\theta_0^{-\alpha}R_k^{2\sigma}<t<0}\|\eta_kw_k\|^2_{L^2(B_{R_k})}\\
&\qquad \quad +\|\nabla(\eta_k w^{\ast}_k)\|^2_{L^2(Q^{\ast}_{R_k}(\theta_0),|y|^{a})}\\
&\qquad \qquad \qquad  \leq \frac{(\delta^{k-1})^{a+1}4^{k-a}(\lambda\omega)^2}{(1+a)R^2}\int^{0}_{-\theta_0^{-\alpha}R_k^{2\sigma}}|A_{l_{k-1},R_{k-1}}(t)|dt\\
&\qquad \qquad \qquad \qquad
+\frac{4^{\sigma(k+2)}\theta_0^{\alpha}(\lambda\omega)^{2}}{2mM^{\alpha}R^{2\sigma}}\int^{0}_{-\theta_0^{-\alpha}R_k^{2\sigma}}|A_{l_k,R_k}(t)|dt.
\end{aligned}
\end{equation}
From this, it follows that
\begin{equation}
\begin{aligned}
&c\left(\frac{4}{A\lambda}\right)^{\alpha}\sup_{-\theta_0^{-\alpha}R_k^{2\sigma}<t<0}
\|\eta_kw_k\|^2_{L^2(B_{R_k})}\\
&\qquad \quad +\theta_0^{\alpha}\|\D (\eta_k w^*_k)\|^2_{L^2(Q^*_{R_k}(\theta_0),|y|^{a})}\\
&\qquad \qquad \qquad \leq (\lambda\omega)^{2}\left(\frac{
(\delta^{k-1})^{a+1}4^{k-a}}{(1+a)R^2}+\frac{4^{\sigma(k+2)}\theta_0^{\alpha}}{2mM^{\alpha}R^{2\sigma}}\right)Z_{k-1}.
\end{aligned}
\end{equation}
Now let us make a change of variable
\begin{equation*}
\tau=\left(\frac{\omega}{A}\right)^{\alpha}t=\theta_0^{\alpha}t.
\end{equation*}
Then, $Q_R(\theta_0)$, $Q^*_R(\theta_0)$ and
$\tilde{Q}^*_R(\theta_0)$ will be transformed to $Q_R(1)$,
$Q^*_R(1)$ and $\tilde{Q}^*_R(1)$ respectively. Let
$\overline{w}(x,\tau)=w(x,\theta_0^{-\alpha}\tau)$ and
$\overline{w}^*(x,y,\tau)=w^*(x,y,\theta_0^{-\alpha}\tau)$. We also
define the quantity $\overline{Z}_k$ to be
\begin{equation*}
\overline{Z}_k=\left|\{(x,t)\in Q_{R_k}(1):
\overline{w}>l_k\}\right|.
\end{equation*}
Then, $\overline{Z}_k$ will be equal to $Z_k$. After the change of
variable, we have
\begin{equation*}
\begin{aligned}
&c\left(\frac{4}{A\lambda}\right)^{\alpha}\sup_{-R_k^{2\sigma}<\tau<0}
\|\eta_k\overline{w}_k\|^2_{L^2(B_{R_k})}
\\
&\qquad \quad +\|\D(\eta_k\overline{w}^*_k)\|^2_{L^2(Q^*_{R_k}(1),|y|^{a})}\\
&\qquad \qquad \qquad \leq (\lambda\omega)^{2}\left(\frac{
(\delta^{k-1})^{a+1}4^{k-a}}{(1+a)R^2}+\frac{4^{\sigma(k+2)}\theta_0^{\alpha}}{2mM^{\alpha}R^{2\sigma}}\right)\overline{Z}_{k-1}.
\end{aligned}
\end{equation*}
Let's choose a small constant $A$ such that
\begin{equation*}
c\left(\frac{4}{A\lambda}\right)^{\alpha}>1.
\end{equation*}
Then
\begin{equation}\label{ineq-local-energy-esti}
\begin{aligned}
&\sup_{-R_k^{2\sigma}<\tau<0}
\|\eta_k\overline{w}_k\|^2_{L^2(B_{R_k})}+\|\D(\eta_k\overline{w}^*_k)\|^2_{L^2(Q^*_{R_k}(1),|y|^{a})}\\
&\qquad \qquad \qquad \leq (\lambda\omega)^{2}\left(\frac{
(\delta^{k-1})^{a+1}4^{k-a}}{(1+a)R^2}+\frac{4^{\sigma(k+2)}\theta_0^{\alpha}}{2mM^{\alpha}R^{2\sigma}}\right)\overline{Z}_{k-1}.
\end{aligned}
\end{equation}
Since
$\eta_k\overline{w}^{\ast}_k\text{1}_{\{0<y<\frac{\delta^{k-1}}{4}\}}$
has the same trace at $y=0$ as $(\eta_k\overline{w}_k)^{\ast}$, we
have
\begin{equation*}
\begin{aligned}
C\int_0^{\frac{\delta^{k-1}}{4}}\int_{\R^n}|\nabla(\eta_k\overline{w}^{\ast}_k)|^2y^a\,dxdy&=\int_0^{\infty}\int_{\R^n}|\nabla(\eta_k\overline{w}^{\ast}_k\text{1}_{\{0<y<\frac{\delta^{k-1}}{4}\}})|^2y^a\,dxdy\\
&\geq
\int_0^{\infty}\int_{\R^n}|\nabla(\eta_k\overline{w}_k)^{\ast}|^2y^a\,dxdy.
\end{aligned}
\end{equation*}
We also have
\begin{equation*}
\begin{aligned}
\int_{Q_{R_k}(1)} |\eta_k \overline{w}_k|^2\, dx\,d\tau
&\geq(l_{k+1}-l_k)^2\int_{-R_k^{2\sigma}}^0|\{(x,t)\in
Q_{R_{k+1}}(1):\overline{w}>l_{k+1}\}|\,d\tau\\
&=\left(\frac{\lambda\omega}{2^{k+3}}\right)^2\overline{Z}_{k+1}.
\end{aligned}
\end{equation*}
Hence, combining above estimates with Sobolev inequalities (Lemma
\ref{eq-Sobolev-Inequality}), the inequality
\eqref{ineq-local-energy-esti} changes into
\begin{equation*}
\begin{aligned}
&\left(\frac{\lambda\,
\omega}{2^{k+3}}\right)^2\overline{Z}_{k+1}\leq \|\eta_k\overline{w}_k\|^2_{L^2(Q_{R_k}(1))}\\
&\quad \leq
C\left[\sup_{-R_k^{2\sigma}\leq
t\leq 0}\|\eta_{k}\,
w_{k}\|^2_{L^{2}(\R^n)}+\|\D(\eta_k\overline{w}^*_k)\|^2_{L^2(Q^*_{R_k}(1),|y|^{a})}\right]\,
\overline{Z}_{k-1}^{\frac{2\sigma}{n+2\sigma}}\\
&\quad \leq C(\lambda\omega)^{2}\left(\frac{
(\delta^{k-1})^{a+1}4^{k-a}}{(1+a)R^2}+\frac{4^{\sigma(k+2)}\theta_0^{\alpha}}{2mM^{\alpha}R^{2\sigma}}\right)\,\overline{Z}_{k-1}^{1+\frac{2\sigma}{n+2\sigma}}
\end{aligned}
\end{equation*}
for some constant $C>0$. Since $0<\sigma,\,\delta<1$, we have
\begin{equation*}
\overline{Z}_{k+1}\leq
C4^{2k}\left(\frac{4^{3-a}}{(1+a)R^2}+\frac{4^{5}\theta_0^{\alpha}}{2mM^{\alpha}R^{2\sigma}}\right)\,\overline{Z}_{k-1}^{1+\frac{2\sigma}{n+2\sigma}}=C'4^{2k}\overline{Z}_{k-1}^{1+\frac{2\sigma}{n+2\sigma}}.
\end{equation*}
Let's choose the constant $N$ to satisfy
\begin{equation*}
N>\sup\left(1, C',
16^{\frac{n+2\sigma}{\sigma}},\frac{4^{n+1}}{\delta^{2}},\frac{2^{4n+9}\omega\|P(1)\|^2_{L^2}}{\lambda^2\delta^{2n}}\right).
\end{equation*}
Then
\begin{equation*}
\left(\frac{N}{16^{\frac{n+2\sigma}{2\sigma}}}\right)^{\frac{2\sigma
k}{n+2\sigma}}\geq N^{\frac{\sigma k}{n+2\sigma}}\geq N^4\geq
C'N^{2\left(1+\frac{2\sigma}{n+2\sigma}\right)}
\end{equation*}
for $k\geq\frac{4(n+2\sigma)}{\sigma}$ and this is equivalent to
\begin{equation*}
N^{-k}\geq
C'4^{2k}N^{-\left(1+\frac{2\sigma}{n+2\sigma}\right)(k-2)}.
\end{equation*}
If we take the constant $\rho$ so sufficiently small that
\begin{equation*}
\overline{Z}_{\overline{k}}\leq N^{-\overline{k}} \qquad \left(1\leq
\overline{k}<\frac{4(n+2\sigma)}{\sigma}+1\right),
\end{equation*}
then \eqref{ieq-z-k} is true for all $k\geq 0$.
\end{proof}
We next assume that the assumption of Lemma (\ref{lem-Hoilder-1})
are violated, i.e., for every sub-cylinder $Q_{R}(\theta_0)$
\begin{equation*}
\left|\left\{(x,t)\in Q_R(\theta_0):
w(x,t)>\mu^+-\frac{\omega}{2}\right\}\right|>\rho\left|Q_{R}(\theta_0)\right|.
\end{equation*}
Since
\begin{equation*}
\mu^+-\frac{\omega}{2}\geq \mu^-+\frac{\omega}{2^{s_0}}, \qquad
\forall s_0\geq 2,
\end{equation*}
we rewrite this as
\begin{equation}\label{eq-measure-mu-}
\left|\left\{(x,t)\in Q_R(\theta_0):
w(x,t)\leq\mu^-+\frac{\omega}{2^{s_0}}\right\}\right|\leq(1-\rho)\left|Q_{R}(\theta_0)\right|
\end{equation}
valid for all cylinders $Q_R(\theta_0)$.
\begin{lemma}\label{lem-violate-1}
\item If \eqref{eq-p-harnack-cond} is violated, then there exists a
time level
\begin{equation*}
t^{\ast}\in
\left[-\theta_0^{-\alpha}R^{2\sigma},-\frac{\rho}{2}\theta_0^{-\alpha}R^{2\sigma}\right]
\end{equation*}
such that
\begin{equation*}
\left|\{x\in B_R;\, w(x,t)\leq
\mu^-+\frac{\omega}{2^{s_0}}\}\right|\,\leq
\frac{1-\rho}{1-\frac{\rho}{2}}\,|B_R|.
\end{equation*}
\end{lemma}
\begin{proof}
If not, for all $t\in
\left[-\theta_0^{-\alpha}R^{2\sigma},-\frac{\rho}{2}\theta_0^{-\alpha}R^{2\sigma}\right]$,
\begin{equation*}
\left|\left\{x\in B_R;\,
w(x,t)\leq\mu^-+\frac{\omega}{2^{s_0}}\right\}\right|>
\frac{1-\rho}{1-\frac{\rho}{2}}\,|B_R|
\end{equation*}
and
\begin{equation*}
\begin{aligned}
&\left|\left\{(x,t)\in Q_R(\theta_0):w(x,t)\leq\mu^-+\frac{\omega}{2^{s_0}}\right\}\right| \\
&\qquad \geq
\int_{-\theta_0^{\alpha}R^{2\sigma}}^{-\frac{\rho}{2}\theta_0^{\alpha}R^{2\sigma}}\left|\left\{x\in
B_R;\, w(x,\tau)\leq\mu^-+\frac{\omega}{2^{s_0}}\right\}\right|\,d\tau\\
&\qquad >(1-\rho)|Q^{\ast}_R(\theta_0)|,
\end{aligned}
\end{equation*}
contradicting (\ref{eq-measure-mu-}).
\end{proof}
The Lemma asserts that at some time level $t^{\ast}$ the set where
$w$ is close to its supremum occupies only a portion of the $B_{R}$.
The next Lemma claims that this indeed occurs for all time levels
near the $Q_R(\theta_0)$. Set
\begin{equation*}
H=\sup_{B^{\ast}_R\times[t^{\ast},0]}\left|\left(w^{\ast}-\left(\mu^-+\frac{\omega}{2^{s_0}}\right)\right)_-\right|.
\end{equation*}
\begin{lemma}\label{lem-w-mu--}
There exists a positive integer $s_1>s_0$ such that if
\begin{equation}\label{eq-H=sup}
H>\frac{\omega}{2^{s_1}},
\end{equation}
then
\begin{equation*}
\left|\left\{x\in B_R;\, w(x,t)\leq
\mu^-+\frac{\omega}{2^{s_1}}\right\}\right|\,\leq
\left(1-\left(\frac{\rho}{2}\right)^2\right)\,|B_R|,
\end{equation*}
for all $t\in[t^{\ast},0]$.
\end{lemma}
\begin{proof}
We introduce the logarithmic function which appears in Section 2 in
\cite{YD1}
\begin{equation*}
\Psi(H,(w^{\ast}-k)_-,c)\equiv \max\left\{{\ln}
\left(\frac{H}{H-(w^{\ast}-k)_-+c}\right) ; 0\right\}
\end{equation*}
for $k=\mu^-+\frac{\omega}{2^{s_0}}$, $c=\frac{\omega}{2^{s_1}}$.
From the definition and the indicated choices we have
\begin{equation*}
\begin{aligned}
&\Psi(H,(w^{\ast}-k)_-,c)\leq (s_1-s_0)\ln 2\\
&|\Psi_{w^{\ast}}(H,(w^{\ast}-k)_-,c)|^2\leq
\left(\frac{2^{s_1}}{\omega}\right)^{2}.
\end{aligned}
\end{equation*}
To simplify the symbolism let us set
\begin{equation*}
\Psi(H,(w^{\ast}-k)_-,c)=\vp(w^{\ast}).
\end{equation*}
We apply to \eqref{eq-Hoilder(M-v)} the testing function
\begin{equation*}
(M-w^{\ast})^{\alpha}\frac{\partial}{\partial
w^{\ast}}\left[\vp^2(w^{\ast})\right]\zeta^2=(M-w^{\ast})^{\alpha}\left[\vp^2(w^{\ast})\right]'\zeta^2
\end{equation*}
where $\zeta(x,z)$ is the smooth cut-off function such that
\begin{equation*}
\begin{aligned}
&\zeta=1 \qquad \mbox{in
$B_{(1-\nu)R}\times(-(1-\nu)R,(1-\nu)R)$},\\
&\zeta=0 \qquad \mbox{on $\partial\{B_R\times(-R,R)\}$}
\end{aligned}
\end{equation*}
and
\begin{equation*}
|D\zeta|\leq \frac{2}{\nu R}.
\end{equation*}
Then, we have for every $t^{\ast}<t<t_0$
\begin{equation*}
\begin{aligned}
0&=-\int_{B_R^{\ast}}y^a
\nabla\left[(M-w^{\ast})^{\alpha}(\vp^2)'\zeta^2\right]\cdot\nabla
w^{\ast}\,\,dxdy\\
&=2\alpha\int_{B_R^{\ast}}y^a\vp\vp'\zeta^2(M-w^{\ast})^{\alpha-1}|\nabla
w^{\ast}|^2\,\,dxdy-4\int_{B_R^{\ast}}y^a\vp\vp'(M-w^{\ast})^{\alpha}\zeta\nabla\zeta\cdot\nabla
w^{\ast}\,\,dxdy\\
&\qquad
-2\int_{B_R^{\ast}}y^a(M-w^{\ast})^{\alpha}\zeta^2(1+\vp)(\vp')^2|\nabla
w^{\ast}|^2\,\,dxdy\\
&\qquad
+\int_{B_R}(M-w)^{\alpha}\zeta^2(\vp(w)^2)'(-\La)^{\sigma}w\,\,dx\\
&\leq
2\alpha\int_{B_R^{\ast}}y^a\vp\vp'\zeta^2(M-w^{\ast})^{\alpha-1}|\nabla
w^{\ast}|^2\,\,dxdy+2\int_{B^{\ast}_R}y^a(M-w^{\ast})^{\alpha}\vp|\nabla
\zeta|^2 \,\,dxdy\\
&\qquad
-2\int_{B_R^{\ast}}y^a(M-w^{\ast})^{\alpha}\zeta^2(\vp')^2|\nabla
w^{\ast}|^2\,\,dxdy-\frac{1}{m}\int_{B_R}\zeta^2(\vp(w)^2)'w_t\,\,dx.
\end{aligned}
\end{equation*}
Since $\vp$ vanishes on the set where $(w^{\ast}-k)_-=0$, we have
\begin{equation*}
\alpha\int_{B_R^{\ast}}y^a\vp\vp'\zeta^2(M-w^{\ast})^{\alpha-1}|\nabla
w^{\ast}|^2\,\,dxdy \leq
\int_{B_R^{\ast}}y^a(M-w^{\ast})^{\alpha}\zeta^2(\vp')^2|\nabla
w^{\ast}|^2\,\,dxdy
\end{equation*}
when $-\frac{\alpha}{e}+1\leq 2^{s_0}$. Hence
\begin{equation*}
\begin{aligned}
&\sup_{t^{\ast}<t<0}\int_{B_R}\Psi^2(H,(w-k)_-,c)(x,0,t)\zeta^2(x,0)\,\,dx\\
&\qquad \leq
\int_{B_R}\Psi^2(H,(w-k)_-,c)(x,0,t^{\ast})\zeta^2(x,0)\,\,dx\\
&\qquad \qquad
+2m\int_{B^{\ast}_R\times[t^{\ast},0]}y^a(M-w^{\ast})^{\alpha}\Psi(H,(w^{\ast}-k)_-,c)|\nabla
\zeta|^2 \,\,dxdyd\tau.
\end{aligned}
\end{equation*}
Combining this with the previous Lemma \ref{lem-violate-1} gives
\begin{equation}\label{eq-PSI-1}
\begin{aligned}
&\int_{B_R}\Psi^2(H,(w-k)_-,c)(x,0,t)\zeta^2(x,0)\,\,dx \\
&\qquad \leq [(s_1-s_0)\ln
2]^2\left(\frac{1-\rho}{1-\frac{\rho}{2}}\right)|B_R|+\frac{R^{2\sigma+a-1}\omega^{\alpha}\left(1-\frac{1}{2^{s_0}}\right)^{\alpha}}{(1+a)\nu^2}(s_1-s_0)\ln
2|B_R|.
\end{aligned}
\end{equation}
The integral on the left hand side of \eqref{eq-PSI-1} is estimated
below by integrating over the smaller set
\begin{equation*}
\left\{x\in B_{(1-\nu)R}:
w(x,t)<\mu^-+\frac{\omega}{2^{s_1}}\right\}.
\end{equation*}
On such a set, since the function $\Psi$ is a decreasing function of
$H$, we find
\begin{equation*}
\Psi^2\left(H,\left(w-\left(\mu^-+\frac{\omega}{2^{s_0}}\right)\right)_-,\frac{\omega}{2^{s_1}}\right)
\geq (s_1-s_0-1)^2\ln^22,
\end{equation*}
and therefore \eqref{eq-PSI-1} gives
\begin{equation*}
\begin{aligned}
&\left|\left\{ x\in B_{(1-\nu)R}:w(x,t)<\mu^-+\frac{\omega}{2^{s_1}}\right\} \right|\\
&\qquad \leq
\left(\frac{1-\rho}{1-\frac{\rho}{2}}\right)\left(\frac{s_1-s_0}{s_1-s_0-1}\right)^2|B_R|+\frac{\gamma}{\nu^2(s_1-s_0-2)}|B_R|
\end{aligned}
\end{equation*}
for constant
$\gamma=\frac{R^{2\sigma+a-1}\omega^{\alpha}\left(1-\frac{1}{2^{s_0}}\right)^{\alpha}}{(1+a)\ln
2}$. On the other hand
\begin{equation*}
\begin{aligned}
&\left|\left\{ x\in
B_{R}:w(x,t)<\mu^-+\frac{\omega}{2^{s_1}}\right\} \right|\\
&\qquad \leq\left|\left\{ x\in
B_{(1-\nu)R}:w(x,t)<\mu^-+\frac{\omega}{2^{s_1}}\right\}
\right|+\left|B_R\bs B_{(1-\nu)R}\right|\\
&\qquad \leq\left|\left\{ x\in
B_{(1-\nu)R}:w(x,t)<\mu^-+\frac{\omega}{2^{s_1}}\right\}
\right|+n\nu\left|B_R\right|.
\end{aligned}
\end{equation*}
Therefore
\begin{equation*}
\begin{aligned}
&\left|\left\{ x\in
B_{R}:w(x,t)<\mu^-+\frac{\omega}{2^{s_1}}\right\} \right|\\
&\qquad \leq
\left[\left(\frac{1-\rho}{1-\frac{\rho}{2}}\right)\left(\frac{s_1-s_0}{s_1-s_0-1}\right)^2+\frac{\gamma}{\nu^2(s_1-s_0-2)}+n\nu\right]|B_R|
\end{aligned}
\end{equation*}
for all $t\in(t^{\ast},0)$.\\
\indent To prove the Lemma, we first choose $\nu$ so small that
$n\nu\leq \frac{3}{8}\rho^2$ and then $s_1$ so large that
\begin{equation*}
\frac{\gamma}{\nu^2(s_1-s_0-2)}\leq \frac{3}{8}\rho^2, \qquad
\left(\frac{s_1-s_0}{s_1-s_0-1}\right)^2\leq
\left(1-\frac{1}{2}\rho\right)(1+\rho).
\end{equation*}
\end{proof}
Well-definedness of $\lim_{y\to 0}y^aw^{\ast}_y$ gives us that there
is no jumping near the hyper-plane $y=0$. More precisely,
\begin{equation*}
w^{\ast}(x,y,t)-w(x,t)=O(y^{1-a}) \qquad a.e. \,\,\, B_R
\end{equation*}
as $y$ goes to zero since the fact that $y^aw^{\ast}_y$ has a limit
as $y\to 0$, immediately implies that
$\lim_{y\to0}\frac{w^{\ast}(x,y,t)-w(x,t)}{y^{1-a}}$ has the same
limit. Hence, we can also make similar estimates of $w^{\ast}$ in
the cube $B^{\ast}_R$ using the measure condition (Lemma
\ref{lem-w-mu--}):
\begin{equation*}
\left|\left\{x\in B_R;\, w(x,t)>
\mu^-+\frac{\omega}{2^{s_1}}\right\}\right|\,\,>
\left(\frac{\rho}{2}\right)^2\,|B_R|>0
\end{equation*}
for all $t\in[t^{\ast},0]$. The conclusion follows.
\begin{lemma}\label{lem-overline-rho}
There exists a constant $\overline{\rho}>0$ such that
\begin{equation*}
\left|\left\{(x,y)\in B^{\ast}_R:w^{\ast}(x,y,t)>
\mu^-+\frac{\omega}{2^{s_1+1}}\right\}\right|>\overline{\rho}|B^{\ast}_R|
\end{equation*}
for all $t\in[t^{\ast},t_0]$.
\end{lemma}
We list Lemma from the literature and adapted here to our situation.
\begin{lemma}[De Giorgi\cite{De}]\label{De-giorgi}
If $f\in W^{1,1}(B_r)$ $(B_r\subset\R^n)$ and $l,k\in \R$, $k<l$,
then
\begin{equation*}
(l-k)\left|\left\{x\in B_r: f(x)>l\right\}\right|\leq
\frac{Cr^{n+1}}{\left|\left\{x\in B_r:
f(x)<k\right\}\right|}\int_{k<f<l}|\nabla f|\,\,dx,
\end{equation*}
where $C$ depends only on $n$.
\end{lemma}
For the remainder of this section we assume that \eqref{eq-H=sup}
holds.
\begin{lemma}\label{lem-Q-ast-small}
\item If \eqref{eq-p-harnack-cond} is violated, for every $\nu_{\ast}\in (0,1)$, there exists a number $s^{\ast}>s_1+1>s_0$ independent of $\omega$
and $R$ such that
\begin{equation*}
\left|\left\{(x,y,t)\in
Q^{\ast}_{R}\left(\theta_0,0,1-\frac{\rho}{2}\right):w^{\ast}(x,y,t)\leq\mu^-+\frac{\omega}{2^{s^{\ast}}}\right\}\right|\leq
\nu_{\ast}\left|Q^{\ast}_{R}\left(\theta_0,0,1-\frac{\rho}{2}\right)\right|.
\end{equation*}

\end{lemma}
\begin{proof}
Apply Lemma \eqref{De-giorgi} over the cube $B^{\ast}_{R}$ for
$f(x,y)=-w^{\ast}(x,y,t)$, $t\in\left(-\frac{\rho
\theta_0^{-\alpha}R^{2\sigma}}{2},0\right)$ and the levels
\begin{equation*}
l=-\mu^--\frac{\omega}{2^{s+1}},\qquad k=-\mu^--\frac{\omega}{2^s},
\qquad s=s_1+1, s_1+2,\cdots,s^{\ast}.
\end{equation*}
Then, from the Lemma \eqref{lem-overline-rho}, we have
\begin{equation}\label{eq-A-s}
\left(\frac{\omega}{2^{s+1}}\right)|A_{s+1}(t)|
\leq\frac{4CR^{n+2}}{\overline{\rho}|B^{\ast}_R|}\int_{A_s(t)\bs
A_{s+1}(t)}|\nabla w^{\ast}|\,\,dxdy
\end{equation}
where
\begin{equation*}
A_s(t)=\left\{(x,y)\in B^{\ast}_R:
w^{\ast}(x,y,t)<\mu^-+\frac{\omega}{2^s}\right\}.
\end{equation*}
Set
\begin{equation*}
A_s=\left\{(x,y,t)\in B^{\ast}_R\times\left[-\frac{\rho
\theta_0^{-\alpha}
R^{2\sigma}}{2},0\right]:w^{\ast}(x,y,t)<\mu^-+\frac{\omega}{2^s}\right\}.
\end{equation*}
From this, integrating \eqref{eq-A-s} over $\left(-\frac{\rho
\theta_0^{-\alpha}R^{2\sigma}}{2},0\right)$ we get
\begin{equation}\label{eq-holder-2th-order}
\begin{aligned}
&\left(\frac{\omega}{2^{s+1}}\right)|A_{s+1}| \leq C'R\iint_{A_s\bs
A_{s+1}}|\nabla
w^{\ast}|y^{\frac{a}{2}}y^{-\frac{a}{2}}\,\,dxdyd\tau\\
&\qquad \leq
C'R\left(\iint_{Q^{\ast}_R\left(\theta_0,0,1-\frac{\rho}{2}\right)}|\nabla
w^{\ast}|^2y^a\,\,dxdyd\tau
\right)^{\frac{1}{2}}\\
&\qquad \qquad
\times\left(\iint_{Q^{\ast}_R\left(\theta_0,0,1-\frac{\rho}{2}\right)}\left(y^{-\frac{a}{2}}\right)^{\frac{4}{1+a}}\,\,dxdyd\tau\right)^{\frac{1+a}{4}}
\left(\iint_{A_s \bs A_{s+1}}\,\,dxdyd\tau\right)^{\frac{1-a}{4}}.
\end{aligned}
\end{equation}
Take the $\frac{4}{1-a}$-th power to obtain
\begin{equation}\label{eq-holder-2th-order}
\left(\frac{\omega}{2^{s+1}}\right)^{\frac{4}{1-a}}|A_{s+1}|^{\frac{4}{1-a}}\leq
C''\left(|A_s|-|A_{s+1}|\right)\left(\iint_{Q^{\ast}_R\left(\theta_0,0,1-\frac{\rho}{2}\right)}|\nabla
w^{\ast}|^2y^a\,\,dxdyd\tau\right)^{\frac{2}{1-a}}
\end{equation}
where
$C''=(C'R)^{\frac{4}{1-a}}\left(\frac{1+a}{1-a}R^{n+\frac{1-a}{1+a}}(t_0-t^{\ast})\right)^{\frac{1+a}{1-a}}$.\\
\indent We estimate the integral on the right hand side by making
use of \eqref{Holder-energy-estimate} written over again as
$Q^{\ast}_{2R}(\theta_0)$, $k=\mu^-+\frac{\omega}{2^{s}}$ and as
$\eta(x,z,t)$ a smooth cut-off function in $Q^{\ast}_{2R}(\theta_0)$
which equals one on $Q^{\ast}_{R}(\theta_0)$, vanishes on the
parabolic boundary of $Q^{\ast}_{2R}(\theta_0)$ and is such that
\begin{equation*}
|\nabla \eta|\leq \frac{1}{R}, \qquad |\eta_t|\leq
\frac{\omega^{\alpha}}{A^{\alpha}R^{2\sigma}}.
\end{equation*}
Then, since $|\nabla w^{\ast}|=|\nabla (w^{\ast}-k)_-|$, we deduce
\begin{equation*}
\begin{aligned}
&\left[\iint_{Q^{\ast}_R\left(\theta_0,0,1-\frac{\rho}{2}\right)}|\nabla
(w^{\ast}-k)_-|^2y^a\,\,dydxd\tau\right]^{\frac{2}{1-a}} \\
&\qquad \qquad \leq
C(A,a,M,n,R,\sigma,\alpha)\left(\frac{\omega}{2^s}\right)^{\frac{4}{1-a}}\left|Q_R\left(\theta_0,0,1-\frac{\rho}{2}\right)\right|^{\frac{2}{1-a}}.
\end{aligned}
\end{equation*}
Substitute this estimate into \eqref{eq-holder-2th-order} and divide
through by $\left(\frac{\omega}{2^{s+1}}\right)^{\frac{4}{1-a}}$.
\begin{equation}\label{eq-A-s-2}
|A_{s+1}|^{\frac{4}{1-a}}\leq
C'''\left|Q^{\ast}_R\left(\theta_0,0,1-\frac{\rho}{2}\right)\right|^{\frac{2}{1-a}}(|A_{s}|-|A_{s+1}|).
\end{equation}
These inequalities are valid for all $s_1+1\leq s\leq s^{\ast}-1$.
Adding \eqref{eq-A-s-2} for $s=s_1+1, s_1+2,\cdots,s^{\ast}-1$, we
have
\begin{equation*}
(s^{\ast}-s_1-2)|A_{s^{\ast}}|^{\frac{4}{1-a}} \leq
C'''\left|Q^{\ast}_R\left(\theta_0,0,1-\frac{\rho}{2}\right)\right|^{\frac{3-a}{1-a}}.
\end{equation*}
To prove the Lemma, we divide $(s^{\ast}-s_1-2)$ and take $s^{\ast}$
so large that
\begin{equation*}
\left(\frac{C'''}{s^{\ast}-s_1-2}\right)^{\frac{1-a}{4}}\frac{1}{\left|Q^{\ast}_R\left(\theta_0,0,1-\frac{\rho}{2}\right)\right|^{\frac{1+a}{4}}}\leq
\nu^{\ast}.
\end{equation*}
\end{proof}
Using the relation between $w$ and $w^{\ast}$, we show next that we
can replace the $Q^{\ast}_R\left(\theta_0,0,1-\frac{\rho}{2}\right)$
and $w^{\ast}$ by $Q_R\left(\theta_0,0,1-\frac{\rho}{2}\right)$ and
$w$ respectively.
\begin{lemma}
\item In addition, we have
\begin{equation*}
\left|\left\{(x,t)\in
Q_R\left(\theta_0,0,1-\frac{\rho}{2}\right):w(x,t)\leq\mu^-+\frac{\omega}{2^{s^{\ast}+1}}\right\}\right|\leq
\nu^{\star}\left|Q_R\left(\theta_0,0,1-\frac{\rho}{2}\right)\right|
\end{equation*}
for every $\nu_{\star}\in (0,1)$.
\end{lemma}
\begin{proof}
for every $t$, $x$ fixed and for
$k=\mu^-+\frac{\omega}{2^{s^{\ast}}}$,
\begin{equation*}
\begin{aligned}
\left(w-k\right)_-&=\left(w^{\ast}-k\right)_-(s)-\int_0^s\partial_y\left[\left(w^{\ast}-k\right)_-\right]\,\,dy\\
&\leq\left(w^{\ast}-k\right)_-(s)+\int_0^s\left|\partial_y\left[\left(w^{\ast}-k\right)_-\right]\right|y^{\frac{a}{2}}\cdot y^{-\frac{a}{2}}\,\,dy\\
&\leq
\left(w^{\ast}-k\right)_-(s)+\frac{s^{1-a}}{1-a}\left(\int_0^s\left|\nabla\left(w^{\ast}-k\right)_-\right|^2y^a\,\,dy\right)^{\frac{1}{2}}.
\end{aligned}
\end{equation*}
So, integrating in $x$ and $t$ over
$Q_R\left(\theta_0,0,1-\frac{\rho}{2}\right)$, we have
\begin{equation*}
\begin{aligned}
& \int_{Q_R\left(\theta_0,0,1-\frac{\rho}{2}\right)}
\left[\left(w-k\right)_-\right]^2\,\,dxdt \\
&\qquad \leq
2\Bigg[\int_{Q_R\left(\theta_0,0,1-\frac{\rho}{2}\right)}
\left[\left(w^{\ast}-k\right)_-(s)\right]^2\,\,dxdt\\
&\qquad \qquad
+\frac{s^{2-2a}}{(1-a)^2}\int_{Q^{\ast}_R\left(\theta_0,0,1-\frac{\rho}{2}\right)}\left|\nabla\left(w^{\ast}-k\right)_-\right|^2y^a\,\,dxdydt\Bigg]
\end{aligned}
\end{equation*}
for any $y=s\leq R$. Hence, integrating in $s$ over $[0,\e_1]$,
\begin{equation*}
\begin{aligned}
&\e_1\int_{A(t)}
\left[\left(w-k\right)_-\right]^2\,\,dxdt\\
&\qquad \leq
2\Bigg[\int_{Q^{\ast}_R\left(\theta_0,0,1-\frac{\rho}{2}\right)}
\left[\left(w^{\ast}-k\right)_-(s)\right]^2\,\,dxdydt\\
&\qquad \qquad
+\frac{{\e_1}^{3-2a}}{(3-2a)(1-a)^2}\int_{Q^{\ast}_R\left(\theta_0,0,1-\frac{\rho}{2}\right)}\left|\nabla\left(w^{\ast}-k\right)_-\right|^2y^a\,\,dxdydt\Bigg].
\end{aligned}
\end{equation*}
By making use of \eqref{Holder-energy-estimate} written over as
$Q^{\ast}_{2R}(\theta_0)$, $k=\mu^-+\frac{\omega}{2^{s}}$ and as
$\eta(x,z,t)$ which is given in the proof of Lemma
\eqref{lem-Q-ast-small}. Then, we obtain
\begin{equation*}
\begin{aligned}
&\iint_{Q^{\ast}_R\left(\theta_0,0,1-\frac{\rho}{2}\right)}|\nabla
(w^{\ast}-k)_-|^2y^a\,\,dydxd\tau \\
&\qquad \qquad \leq
C(A,a,M,n,R,\sigma,\alpha)\left(\frac{\omega}{2^{s^{\ast}}}\right)^{2}\left|Q_R\left(\theta_0,0,1-\frac{\rho}{2}\right)\right|.
\end{aligned}
\end{equation*}
Thus, we have
\begin{equation*}
\begin{aligned}
&\e_1\left(\frac{\omega}{2^{s^{\ast}+1}}\right)^2\left|\left\{(x,t)\in
 Q_R\left(\theta_0,0,1-\frac{\rho}{2}\right):w \leq
\mu^-+\frac{\omega}{2^{s^{\ast}+1}}\right\}\right|\\
&\qquad \leq \e_1\int_{A(t)}
\left[\left(w-k\right)_-\right]^2\,\,dxdt \\
&\qquad \qquad \leq
C\left(\frac{\omega}{2^{s^{\ast}}}\right)^2\left(\nu^{\ast}+\frac{{\e_1}^{3-2a}}{(3-2a)(1-a)^2}\right)\left|Q_R\left(\theta_0,0,1-\frac{\rho}{2}\right)\right|.
\end{aligned}
\end{equation*}
To prove the lemma, we take $\nu^{\ast}$ and $\e_1$ so small that
\begin{equation*}
4C\left(\frac{\nu^{\ast}}{\e_1}+\frac{{\e_1}^{2-2a}}{(3-2a)(1-a)^2}\right)\leq
\nu^{\star}
\end{equation*}
for $\nu^{\star}$ in $(0,1)$.
\end{proof}
Now we show that given $\nu^{\star}$, it determines a level
$\mu^-+\frac{\omega}{2^{s^{\ast}+1}}$ and a cylinder so that the
measure of the set where $w$ is below such a level can be made
smaller than $\nu^{\star}$, on that particular cylinder. Hence, for
sufficiently small number $\nu^{\star}$, we have new powerful
assumption like to that of Lemma \eqref{lem-Hoilder-1}. Therefore,
using the same arguments as in Lemma \eqref{lem-Hoilder-1} with
$\left(w-\left(\mu^-+\frac{\omega}{2^{s^{\ast}+1}}\right)\right)_-$
we can obtain the following result.
\begin{lemma}\label{lem-lower-bound}
The number $\nu^{\star}$ (and hence $s^{\ast}$) can be chosen so
that
\begin{equation*}
w(x,t)\geq \mu^-+\frac{\lambda\omega}{2^{s^{\ast}+2}} \qquad \qquad
\mbox{a.e.}\quad
Q_{\frac{R}{2}}\left(\theta_0,0,\frac{\rho}{2}\right)
\end{equation*}
for some $\lambda$ in $\left(0,\frac{1}{2}\right)$.
\end{lemma}

\begin{lemma}[Oscillation Lemma]\label{lem-oscillation}
There exist constants $\lambda^{\ast}>0$ and $\kappa\in(0,1)$ such
that if
\begin{equation*}
\osc_{Q^{\ast}_{2R}}w^{\ast}=\omega=\mu^+-\mu^-,
\end{equation*}
then
\begin{equation*}
\osc_{Q^{\ast}_{\frac{R}{4}}\left(\theta_0,0,1-\frac{\rho}{2}\right)}w^{\ast}\leq
\omega-\lambda^{\ast}=\kappa\omega.
\end{equation*}
\end{lemma}
\begin{proof}
We first suppose that
\begin{equation*}
\left|\left\{(x,t)\in
  Q_R(\theta_0); w(x,t)>\mu^+-\frac{\omega}{2}\right\}\right|<\varepsilon|Q_R(\theta_0)|
\end{equation*}
for sufficiently small $\varepsilon>0$. Then, from the Lemma
\eqref{lem-Hoilder-1}, we obtain
\begin{equation*}
w(x,t)<\mu^+-\frac{\lambda\omega}{4} \qquad \qquad \mbox{in}\quad
Q_{\frac{R}{2}}(\theta_0)\left(\supset
Q_{\frac{R}{2}}\left(\theta_0,0,1-\frac{\rho}{2}\right)\right).
\end{equation*}
Let's consider the function $h_3$ defined by:
\begin{equation*}
\begin{cases}
\begin{aligned}
&\nabla(y^a\nabla h_3)=0 \qquad\qquad \mbox{in
$B^{\ast}_{\frac{R}{2}}$}\\
&\quad h_3=\mu^+ \qquad \qquad \mbox{on $\partial
B^{\ast}_{\frac{R}{2}}\cap \{y>0\}$}\\
& h_3=\mu^+-\frac{\lambda\omega}{4} \qquad \qquad \mbox{on $y=0$}.
\end{aligned}
\end{cases}
\end{equation*}
Then, we have
\begin{equation*}
w^{\ast}\leq h_3 \qquad \qquad \mbox{in
$Q^{\ast}_{\frac{R}{2}}\left(\theta_0,0,1-\frac{\rho}{2}\right)$}
\end{equation*}
from the maximum principle. Since $h_3\leq \mu^+-\lambda_1$ in
$B^{\ast}_{\frac{R}{4}}$ for some
$0<\lambda_1<\frac{\lambda\omega}{4}$, we obtain that
\begin{equation*}
\osc_{Q^{\ast}_{\frac{R}{4}}\left(\theta_0,0,1-\frac{\rho}{2}\right)}w^{\ast}\leq
\mu^+-\lambda_1-\mu^-=\omega-\lambda_1.
\end{equation*}
Next, we assume that
\begin{equation*}
\left|\left\{(x,t)\in
  Q_R(\theta_0); w(x,t)>\mu^+-\frac{\omega}{2}\right\}\right|>\varepsilon_1|Q_R(\theta_0)|
\end{equation*}
for some $\varepsilon_1>0$. From the Lemma \eqref{lem-lower-bound},
we have
\begin{equation*}
w(x,t)\geq \mu^-+\frac{\lambda\omega}{2^{s^{\ast}+2}} \qquad \qquad
\mbox{in}\quad
Q_{\frac{R}{2}}\left(\theta_0,0,\frac{\rho}{2}\right).
\end{equation*}
We also consider the function $h_4$ defined by:
\begin{equation*}
\begin{cases}
\begin{aligned}
&\nabla(y^a\nabla h_4)=0 \qquad\qquad \mbox{in
$B^{\ast}_{\frac{R}{2}}$}\\
&\quad h_4=\mu^- \qquad \qquad \mbox{on $\partial
B^{\ast}_{\frac{R}{2}}\cap \{y>0\}$}\\
& h_4=\mu^-+\frac{\lambda\omega}{2^{s^{\ast}+2}} \qquad \qquad
\mbox{on $y=0$}.
\end{aligned}
\end{cases}
\end{equation*}
Then, we have
\begin{equation*}
w^{\ast}\geq h_4 \qquad \qquad \mbox{in
$Q^{\ast}_{\frac{R}{2}}\left(\theta_0,0,1-\frac{\rho}{2}\right)$}
\end{equation*}
from the minimum principle. Since $h_4\geq \mu^-+\lambda_2$ in
$B^{\ast}_{\frac{R}{4}}$ for some $0<\lambda_2\leq
\frac{\lambda\omega}{2^{s^{\ast}+2}}$, we have
\begin{equation*}
\osc_{Q^{\ast}_{\frac{R}{4}}\left(\theta_0,0,1-\frac{\rho}{2}\right)}w^{\ast}\leq
\mu^+-(\mu^-+\lambda_1)=\omega-\lambda_1.
\end{equation*}
By taking $\lambda^{\ast}=\min\{\lambda_1,\lambda_2\}$, we get a
desired conclusion
\begin{equation*}
\osc_{Q^{\ast}_{\frac{R}{4}}\left(\theta_0,0,1-\frac{\rho}{2}\right)}w^{\ast}\leq
\omega-\lambda^{\ast}=\kappa\omega.
\end{equation*}
\end{proof}
\begin{thm}[H\"older estimates]\label{thm-Hoilder}
There exist constants $\gamma>1$ and $\beta\in(0,1)$ that can be
determined a priori only in terms of the data, such that for all the
cylinders
\begin{equation*}
\osc_{Q^{\ast}_{r}\left(\theta_0,0,1-\frac{\rho}{2}\right)}w^{\ast}\leq
\gamma\omega\left(\frac{r}{R}\right)^{\beta} \qquad \qquad (0<r\leq
R).
\end{equation*}
\end{thm}
\begin{proof}
From the Oscillation Lemma (Lemma \eqref{lem-oscillation}), we
obtain
\begin{equation*}
\osc_{Q^{\ast}_{\frac{R}{2^k}}\left(\theta_0,0,1-\frac{\rho}{2}\right)}w^{\ast}\leq
\kappa^k\omega.
\end{equation*}
Let now $0<r\leq R$ be fixed. There exists a non-negative integer
$k$ such that
\begin{equation*}
\frac{R}{2^{k+1}}\leq r\leq \frac{R}{2^k}.
\end{equation*}
This implies the inequalities
\begin{equation*}
-\log_2\left(\frac{r}{R}\right)\leq k+1
\end{equation*}
and
\begin{equation*}
\kappa^n\leq
\frac{1}{\kappa}\kappa^{-\log_2\left(\frac{r}{R}\right)}=\frac{1}{\kappa}\left(\frac{r}{R}\right)^{-\log_2\kappa}.
\end{equation*}
Since $\lambda^{\ast}\leq \frac{\lambda\omega}{2^{s^{\ast}+2}}$, we
have $\frac{1}{2}<\kappa<1$. Hence
\begin{equation*}
\osc_{Q^{\ast}_{\frac{R}{2^k}}\left(\theta_0,0,1-\frac{\rho}{2}\right)}w^{\ast}\leq
\gamma\omega\left(\frac{r}{R}\right)^{\beta}
\end{equation*}
where $\gamma=\frac{1}{\kappa}>1$ and $0<\beta=-\log_2\kappa<1$. To
conclude the proof we observe that the cylinder
$Q^{\ast}_{r}\left(\theta_0,0,1-\frac{\rho}{2}\right)$ is included
in
$Q^{\ast}_{\frac{R}{2^k}}\left(\theta_0,0,1-\frac{\rho}{2}\right)$.
\end{proof}
We finish with the proof of Theorem \ref{thm-main}. \\
\begin{proof}[\textbf{Proof of Theorem \ref{thm-main}}]
From the Theorem \eqref{thm-Hoilder}, the solution $v^{\ast}$ of the
problem \eqref{eq-Holder} satisfies
\begin{equation*}
\osc_{Q^{\ast}_{r}\left(\theta_0,0,1-\frac{\rho}{2}\right)}v^{\ast}=\osc_{Q^{\ast}_{r}\left(\theta_0,0,1-\frac{\rho}{2}\right)}w^{\ast}\leq
\gamma\omega\left(\frac{r}{R}\right)^{\beta} \qquad (0<r\leq R)
\end{equation*}
since $v^{\ast}=M-w^{\ast}$. This gives that $v^{\ast}$ is
$C^{\beta}$ at $(x,0,t)$, and so $v$ is $C^{\beta}$ at $(x,t)$.
\end{proof}

\section{Asymptotic behaviour for the FDE with fractional powers}

\setcounter{equation}{0}
\setcounter{thm}{0}

\subsection{Special solutions and stabilization}
The asymptotic description is based on the existence of appropriate
solutions that serve as model for the behavior near extinction:
there is a self-similar solution of the form
\begin{equation}\label{eq-frac-FDE-self}
U(x,t;T)=(T-t)^{1/(1-m)}f(x)
\end{equation}
for a certain profile $f>0$, where $\vp=f^m$ is the solution of the
super-linear elliptic equation
\begin{equation*}
(-\La)^{\sigma}\vp(x)=\frac{1}{1-m}\vp(x)^p, \qquad p=\frac{1}{m}
\end{equation*}
such that $\vp>0$ in $\Omega$ with zero on $\R^n\bs\Omega$. Hence,
similarity means in this case the separate-variables form. The
existence and regularity of this solution depends on the exponent
$p$, indeed it exists for $p<(n+2\sigma)/(n-2\sigma)$, the Sobolev
exponent. Since $p=1/m$, this means that smooth separate-variables
solutions exist for
\begin{equation*}
\frac{n-2\sigma}{n+2\sigma}<m<1
\end{equation*}
an assumption that will be kept in the sequel. Note that the family
of solutions (\ref{eq-frac-FDE-self}) has a free parameter $T>0$.
\subsection{Stabilization}
The above family of solutions allows to describe the behavior of
general solutions near their extinction time.
\begin{lemma}\label{lem-fde-lower-bound}
Let $u$ be a non-negative solution of \eqref{eq-frac-FDE}. Suppose
that the extinction time $T^{\ast}=10$. Then
\begin{equation*}
\sup_{\Omega}v(x,2)=\sup_{\Omega}u^m(x,2)\geq c_0
\end{equation*}
where $c_0$ is a uniform constant.
\end{lemma}
\begin{proof}
Suppose that $\sup_{\Omega}v(x,2)<\e$ for sufficiently small $\e>0$.
Then, we have
\begin{equation*}
\left(\int_{\R^n}v^{\frac{m+1}{m}}(x,2)\,\,dx\right)^{\frac{1-m}{1+m}}\leq
\e^{\frac{1-m}{m}}|\Omega|^{\frac{1-m}{1+m}}.
\end{equation*}
Hence, from the estimates on Finite Extinction Time, (Lemma
\ref{lem-Estimates-Finite-Extinction-Time}),
\begin{equation*}
T^{\ast}-2\leq C \e^{\frac{1-m}{m}}|\Omega|^{\frac{1-m}{1+m}}.
\end{equation*}
If $\e>0$ is small enough, then we have
\begin{equation*}
T^{\ast}\leq 5,
\end{equation*}
which is a contradiction to the assumption $T^{\ast}=10$.
\end{proof}
Using the same arguments as in Section 5 in \cite{KL}, we can obtain
the following two Corollaries.
\begin{cor}\label{cor-lowerbound-function}
Let $u$ be a nonnegative solution of \eqref{eq-frac-FDE} and
$u_0^{m}$ be a super-solution for the fractional power of the
Laplacian. Suppose that the extinction time $T^{\ast}=10$. There
exist a function $\psi(x)>0$ in $\Omega$ and a constant $r\in(0,1)$
such that
\begin{equation*}
v(x,t)\geq C\psi(x) \qquad \mbox{for $2-r\leq t\leq 2$}
\end{equation*}
where $C$ is uniform.
\end{cor}
\begin{proof}
Combining the Lemma \eqref{lem-fde-lower-bound} with Theorem
\eqref{thm-Hoilder}, we can take a point $x_0\in\Omega$ and
constants $\varepsilon, \sigma\in (0,1)$ and $\tilde{r}>0$ such that
\begin{equation*}
v(x,t)\geq c_1>0
\end{equation*}
for $(x,t)\in
B_{\tilde{r}}(x_0)\times[2-\varepsilon\tilde{r}^{2\sigma},2]$. Let
$\psi$ be the solution of
\begin{equation*}
\begin{cases}
\begin{aligned}
(-\La)^{\sigma}&\psi(x)=0 \qquad \qquad \mbox{in
$\Omega\bs\overline{B_{\tilde{r}}(x_0)}$}\\
\psi&=1 \qquad \qquad \qquad \mbox{on $B_{\tilde{r}}(x_0)$}\\
\psi&=0 \qquad \qquad \qquad \mbox{on $\R^n\bs\Omega$}.
\end{aligned}
\end{cases}
\end{equation*}
On the other hand, $v_t$ satisfies
\begin{equation*}
\begin{cases}
\begin{aligned}
mv^{1-\frac{1}{m}}(-\La)^{\sigma}v_t&+(v_{t})_t=\frac{m-1}{mg}v_t^2\qquad\,\,
\mbox{in $\Omega$}\\
v_t&=0 \qquad \qquad \qquad \qquad \mbox{on $\R^n\bs\Omega$}.
\end{aligned}
\end{cases}
\end{equation*}
Since the constant function $f=0$ is a solution of above problem and
$(-\La)^{\sigma}u_0^m\geq 0$, we have $(-\La)^{\sigma}v(x,t)\geq 0$
for all $0<t<T^{\ast}$. Hence
\begin{equation*}
c_1\psi(x,t)\leq v(x,t) \qquad \qquad \mbox{in
$\Omega\times[2-\varepsilon\tilde{r}^{2\sigma},2]$}.
\end{equation*}
Therefore, we have
\begin{equation*}
C\psi(x)\leq v(x,t) \qquad \qquad \mbox{in $\Omega\times[2-r,2]$}
\end{equation*}
where $C=c_1$ and $r=\varepsilon\tilde{r}^{2\sigma}$.
\end{proof}
\begin{cor}\label{cor-sub-lowerbound-function}
Let $u$ be a nonnegative solution of \eqref{eq-frac-FDE} and
$u_0^{m}$ be a super-solution for the fractional power of the
Laplacian with the extinction time $T^{\ast}$. For $0<t<T^{\ast}$
\begin{equation*}
u\left(x,\frac{T^{\ast}-t}{10}s+t\right)\geq
C\psi(x)(T^{\ast}-t)^{\frac{1}{1-m}} \qquad (2-\delta\leq s\leq 2).
\end{equation*}
\end{cor}
\begin{proof}
One can easily check that the scaled function $\tilde{g}$ defined by
\begin{equation*}
\tilde{v}(x,t)=\left(\frac{T^{\ast}-t_0}{10}\right)^{\frac{-m}{1-m}}v\left(x,\left(\frac{T^{\ast}-t_0}{10}\right)t+t_0\right)
\end{equation*}
is a solution of \eqref{eq-nlocal-PME-1} with the finite extinction
time $10$. Then, there exists a function $\psi(x)>0$ such that
\begin{equation*}
C\psi(x)\leq
\tilde{v}(x,\tilde{t})=\left(\frac{T^{\ast}-t_0}{10}\right)^{\frac{-m}{1-m}}v\left(x,\frac{T^{\ast}-t_0}{10}\tilde{t}+t_0\right)
\qquad (2-r\leq \tilde{t} \leq 2)
\end{equation*}
from the Corollary \eqref{cor-lowerbound-function}. Substituting in
this inequality $v=u^m$, we easily come to a conclusion.
\end{proof}
In the end of this section, we look at an asymptotic behavior of a
sequence of time-slice for normalized solutions. We first put
\begin{equation*}
u(x,t)=(T^{\ast}-t)^{\frac{1}{1-m}}\overline{u}(x,\tau) \qquad
\qquad \left(\tau=\ln\frac{T^{\ast}}{T^{\ast}-t}\right).
\end{equation*}
Then, the problem \eqref{eq-frac-FDE} is mapped into:
\begin{equation}\label{eq-normalize-fde}
\begin{cases}
\begin{aligned}
\frac{\overline{u}}{1-m}=&(-\La)^{\sigma}\overline{u}^m+\overline{u}_{\tau}
\qquad \qquad \qquad \mbox{in $\Omega$}\\
\overline{u}(x,\tau)&=0 \qquad \qquad \qquad \qquad \quad \mbox{on
$\R^n\bs\Omega$}\\
\overline{u}(x,0)=&\left(\frac{1}{T^{\ast}}\right)^{\frac{1}{1-m}}u_0(x)
\qquad \qquad \mbox{in $\Omega$}.
\end{aligned}
\end{cases}
\end{equation}
Observe that the new time $\tau$ ranges from $0$ to $\infty$. We now
state the main Theorem in this section.
\begin{thm}
Under the above assumptions on $u_0$ and $m$, we have the following
property near the extinction time of a solution $u(x,t)$: for any
sequence $\{u(x,t_n)\}$, we have a subsequence $t_{n_k}\to T^{\ast}$
and a $\vp(x)$ such that
\begin{equation*}
\lim_{k\to\infty}(T^{\ast}-t_{n_k})^{-1/(1-m)}\big|u(x,t_{n_k})-U(x,t_{n_k};T^{\ast})\big|\to
0
\end{equation*}
uniformly in compact subset of $\Omega$ for
$U(x,t;T^{\ast})=(T^{\ast}-t)^{1/(1-m)}\vp^{1/m}(x)$ where $\vp$ is
a eigen-function of fully nonlinear equation
\begin{equation*}
\begin{cases}
\begin{aligned}
(-\La)^{\sigma}\vp &=\frac{1}{1-m}\vp^{\frac{1}{m}} \qquad \mbox{in
$\Omega$}\\
\vp &=0 \qquad \qquad \mbox{on $\R^n\bs\Omega$}\\
\vp &>0 \qquad \qquad \mbox{in $\Omega$}.
\end{aligned}
\end{cases}
\end{equation*}
\end{thm}
\begin{proof}
The function $\overline{v}=\overline{u}^m$ satisfies the equation
\begin{equation}\label{eq-overline-v}
\frac{\overline{v}^{\frac{1}{m}}}{1-m}=(-\La)^{\sigma}\overline{v}+\frac{1}{m}\overline{v}^{\frac{1-m}{m}}\overline{v}_{\tau}
\end{equation}
in $\Omega\times(0,\infty)$, with $\overline{v}=0$ on
$\R^n\bs\Omega$. Now define the functional
\begin{equation*}
F(f)=\int_{\R^n}\left(\frac{1}{2}f(-\La)^{\sigma}f-\frac{m}{(1-m)(1+m)}f^{\frac{1+m}{m}}\right)\,dx
\end{equation*}
and $g(\tau)=F(\overline{v}(\cdot,\tau))$. Then a simple calculation
yields
\begin{equation*}
g'(\tau)=-\frac{1}{m}\int_{\R^n}\overline{v}^{\frac{1-m}{m}}(x,\tau)\overline{v}^2_{\tau}(x,\tau)\,dx,
\end{equation*}
the right side being non-positive since $\overline{v}\geq 0$. Lemma
\ref{lem-lemma-2} shows that
$\int_{\R^n}\overline{v}^{\frac{1+m}{m}}\,dx$ bounded in $\tau$, so
$g(\tau)$ is bounded below. Therefore $\lim_{\tau\to\infty}g(\tau)$
exists and there exists a sequence of times $\tau_n\to\infty$ such
that $g'(\tau_n)\to 0$.\\
\indent From the Lemma \ref{lem-fde-extinction}, we now translate
the estimate information in terms of $\overline{v}$ to the estimate
\begin{equation*}
0\leq \overline{v}\leq C
\end{equation*}
where $C>0$ is a universal constant. In addition, for each $\tau$,
$\overline{v}(\cdot,\tau_n)$ is equicontinuous in every compact set
$K\subset\Omega$ from the Lemma \eqref{thm-Hoilder}. Hence, every
subsequence, again labeled $\tau_n$,
$\{\overline{v}(\cdot,\tau_n)\}$ has a subsequence
$\{\overline{v}(\cdot,\tau_{n_k})\}$ that converges to some function
$\vp(x)$ uniformly on every compact subset of $K$. Also its limit is
non-trivial since $\overline{v}(x,\tau)\geq C\psi(x)$ when
$\tau>\ln\frac{5}{4}$. Note also, by Lemma \ref{lem-lemma-2}, that
$\int_{\R^n}\overline{v}^{\frac{1+m}{m}}(x,\tau)\,dx$ is monotone
decreasing, hence
\begin{equation*}
\lim_{\tau\to\infty}\int_{\R^n}\overline{v}^{\frac{1+m}{m}}(x,\tau)\,dx=\int_{\R^n}\vp^{\frac{1+m}{m}}(x)\,dx.
\end{equation*}
Multiply equation \eqref{eq-overline-v} by any test function
$\eta\in C^{\infty}_c(K)$ and integrate in space, $x\in K$. Then,
for each $\tau_{n_k}$,
\begin{equation}\label{eq-stability}
\frac{1}{m}\int_{K}\overline{v}^{\frac{1-m}{m}}\overline{v}_{\tau}\eta\,dx=\int_{K}\overline{v}\big[-(-\La)^{\sigma}\eta\big]\,dx+\frac{1}{1-m}\int_{K}\overline{v}^{\frac{1}{m}}\eta\,dx
\end{equation}
from the Theorem 1.2 in \cite{Gu}(Integration by part). Since the
absolute value of the left hand side of \eqref{eq-stability} is
bounded above by
\begin{equation}\label{eq-bound-1}
\left(\int_{\R^n}\eta^{\frac{2(1+m)}{m}}\,dx\right)^{\frac{m}{2(m+1)}}\left(\int_{\R^n}\overline{v}^{\frac{(1-m)(1+m)}{m}}\,dx\right)^{\frac{1}{2(m+1)}}\left(\int_{\R^n}\overline{v}^{\frac{1-m}{m}}\overline{v}^2_{\tau}\,dx\right)^{\frac{1}{2}}
\end{equation}
and the third term of \eqref{eq-bound-1} has limit zero as
$\tau_{n_k}\to \infty$, we get in the limit $\tau_{n_k}\to\infty$
\begin{equation*}
\int_{K}\vp(-\La)^{\sigma}\eta\,dxd=\frac{1}{1-m}\int_{K}\vp^{\frac{1}{m}}\eta\,dx.
\end{equation*}
which is weak formulation of the equation
\begin{equation}\label{eq-limit-pro-fde}
(-\La)^{\sigma}\vp=\frac{1}{1-m}\vp^{\frac{1}{m}}\qquad \mbox{in
$K$}.
\end{equation}
By the arbitrary choice of a compact subset $K$ in $\Omega$,
\eqref{eq-limit-pro-fde} holds in $\Omega$.
\end{proof}

{\bf Acknowledgement} Ki-Ahm Lee was supported by the Korea Research Foundation Grant funded by the Korean Government(MOEHRD, Basic Research Promotion Fund)( KRF-2008-314-C00023).

\end{document}